\documentclass[11pt,a4paper]{article}
\usepackage{amsfonts}
\usepackage{physics,amsmath,amssymb,amsthm,amsfonts,xcolor,graphicx,units,xfrac}
\usepackage{amsmath,amssymb}
\usepackage{mathrsfs}
\usepackage{hyperref}
\usepackage{graphicx}
\usepackage{float}

\newtheorem{theorem}{Theorem}[section]
\newtheorem{prop}[theorem]{Proposition}
\newtheorem{eg}[theorem]{Examples}

\newtheorem{lemma}[theorem]{Lemma}

\newtheorem{rem}[theorem] {Remark}

\newtheorem{cor}[theorem]{Corollary}

\numberwithin{equation}{section}\allowdisplaybreaks

\def\leq{\leqslant}

\def\leq{\leqslant}
\def\geq{\geqslant}

\begin{document}

\title{\large\bf  GLOBAL ATTRACTOR FOR WEAKLY DAMPED, FORCED mKdV EQUATION BELOW ENERGY SPACE}

\author{\normalsize \bf PRASHANT GOYAL   \\
}
\date{}
 \maketitle

\AtEndDocument{\bigskip{ Goyal Prashant \\ Department of Mathematics \\ Graduate School of Science \\ Kyoto University \\ Kyoto 606-8501 \\ Japan} \\ \texttt{goyalprashant194@gmail.com}
\bigskip

}

\thispagestyle{empty}
\begin{abstract}

We prove the existence of the global attractor in $\dot H^s$, $s > 11/12$ for the weakly damped and forced mKdV on the one dimensional torus.
The existence of global attractor below the energy space has not been known, though the global well-posedness below the energy space is established.
We directly apply the I-method to the damped and forced mKdV, because the Miura transformation does not work for the mKdV with damping and forcing terms.
We need to make a close investigation into the trilinear estimates involving resonant frequencies, which are different from the bilinear estimates corresponding to the KdV.

%
\end{abstract}


\section{Introduction}
We consider the modified Korteweg-de Vries (in short, mKdV) equation:
\begin{align}    \label{intro10}
&\partial_{t}u + \partial_{x}^{3}u \pm 2\partial_{x} u^{3} +\gamma u = f,\hspace{2.5mm} t>0, \hspace{1.5mm} x \in \mathbb{T}, \\
&u(x,0) = u_{0}(x) \in \dot H^{s}(\mathbb{T}), \label{intro9}
\end{align}
where $\mathbb{T}$ is the one-dimensional torus, $\gamma >0$ is the damping parameter and $f \in \dot H^{1}(\mathbb{T})$ is the external forcing term which does not depends on $t.$ In equation (\ref{intro10}), ``$+$" and ``$-$" represent the focussing and defocussing cases, respectively. We consider the inhomogeneous Sobolev spaces $H^{s} = \{f\hspace{2mm} | \sum_{k \in \mathbb{Z}} \langle k \rangle^{2s}|\hat{f}(k)|^{2}  <  \infty \}$ where $\langle \cdot \rangle = (1+|\cdot|)$ and the homogeneous Sobolev spaces $\dot H^s = \{ f \in H^s| \hat{f}(0) = 0 \}$. The mKdV equation models the propagation of nonlinear water waves in the shallow water approximation. We only consider the focussing case as the defocussing case follows with the same assertion. Also, considering inhomogeneous Sobolev norm is very important as for homogeneous Sobolev norm, Proposition \ref{TL Main result} does not hold for more details (see appendix by Nobu Kishimoto). From the arguments in \cite{G02}, \cite{G01} and \cite{G03}, the existence of global attractor for equations (\ref{intro10})-(\ref{intro9}) directly follows for $s\geq 1$ in $H^{s}$. In the present paper, we prove the existence of global attractor below the energy space in $\dot {H}^{s}(\mathbb{T})$ for $1 > s >11/12.$ \par 

Miura \cite{M01},\cite{M02} and \cite{ME} studied the properties of solutions to the Korteweg-de Vries (KdV) equation and its generalization. Miura \cite{M01} established the Miura transformation between the solutions of mKdV and KdV. Indeed, if $u$ satisfies equation (\ref{intro10}) with $``+"$ sign, then the function defined by
 $$p = \partial_{x}u + iu^{2}$$ 
satisfies the KdV equation, where $i = \sqrt{-1}$. Colliander, Keel, Staffilani, Takaoka and Tao \cite{CKSTT02} presented the $I$-method and proved the existence of global solution for mKdV in the Sobolev space $H^{s}(\mathbb{T})$ for $s \geq 1/2$ by using the Miura transformation. However, the Miura transformation does not work well for the weakly damped and forced mKdV.  In fact, if we consider the mKdV and KdV equations with the damping and forcing term and apply the Miura transformation, we get
\begin{align}     \label{intro3}
p_{t} + p_{xxx} - 6ipp_{x} + \gamma p = (2iu + \partial_{x})f - i\gamma  u^{2}.
\end{align}
It is clear from (\ref{intro3}) that the Miura transformation does not
transform the solution of mKdV equation to the solution of KdV equation.
For this reason, the results of damped and forced KdV can not be directly converted to those of damped and forced mKdV by the Miura transform unlike the case without damping and forcing terms.
\par

The study of global attractor is important as it characterizes the global behaviour of all solutions. 
The asymptotic behaviour of solutions below the energy space has not been known, though the global well-posedness below the energy space is already proved for the Cauchy problem of \eqref{intro10}-\eqref{intro9}.
To study the asympototic behaviour of the solution of mKdV equation below energy space, we need to study the global attractor below energy space. Chen, Tian and Deng \cite{CLX01} used Sobolev inequalities and \textit{a priori} estimates on $u_{x},u_{xx}$ derived by the energy method to show the existence of global attractor in $H^{2}.$ Dlotko, Kania and Yang \cite{TKY} considered more generalized KdV equation and showed the existence for global attractor in $H^{1}.$ 
It is instructive to look at known results on KdV, since KdV has been more extensively studied than mKdV.
Tsugawa \cite{T} proved the existence of global attractor for KdV equation in $\dot{H}^{s}$ for $0 >s >-3/8$  by using the $I$-method. Later, Yang \cite{X} closely investigated Tsugawa's argument to bring down the lower bound from $s >-3/8$ to $s \geq -1/2$.\par

Though mKdV has many common properties with KdV, there is a big difference between KdV and mKdV in the structure of resonance.
For KdV, we consider the homogeneous Sobolev spaces instead of the inhomogeneous one, which eliminates the resonant frequencies in quadratic nonlinearity (see Bourgain \cite{B}).
On the other hand, for the homogeneous mKdV equation, to eliminate the resonant frequencies in cubic nonlinearity, we need to consider the reduced equation (or the renormalized equation) 
\begin{align}   \label{intro 4}
 \partial_{t}u + \partial_{x}^{3}u + 6\left(u^{2} - \frac{1}{2\pi}\|u\|^{2}_{L^{2}}\right)\partial_{x}u = 0.
\end{align}
Without damping and forcing terms, the $L^2$ norm of the solution is conserved.
So, the transformation from the original mKdV eqation to the reduced mKdV equation is just the translation with constant velocity.
But this is not the case with damped and forced mKdV.
The resonant structure of cubic nonlinearity is quite different from that of quadratic nonlinearity. 
Therefore, in the mKdV case, we need to directly handle the resonant trilinear estimate as well as the non-resonant trilinear estimate.
In this respect, it seems difficult to employ the modified energy similar to that used in \cite{T},\cite{X}. Especially, the scaling argument is one of the main ingredient of the $I$ method. So we need to make the dependence of estimates on the scaling parameter $\lambda$ also.
Hence, the following questions naturally arise: How should we treat the nonlinearity of mKdV equation with the damping and forcing terms? When we can not use Miura transformation, how should we treat mKdV equation ? To deal with such issues, we apply the $I$-method directly to \eqref{intro1}-\eqref{intro2} in the present paper and prove the following result: 
\begin{theorem}  \label{intro theorem}
Assume $11/12 < s < 1$ and $u_{0} \in \dot H^{s}$. Let $S(t)$ is the semi-group generated by the solution of mKdV. Then, there exists two operators $L_{1} (t)$ and $L_{2} (t)$ such that  
\begin{align*}
&S(t) u_{0} = L_{1}(t)u_{0} + L_{2}(t) u_{0}, \\
&\sup \limits_{t>T_{1}} \|L_{1}(t)u_{0}\|_{H^{1}} < K,  \\  
&\|L_{2}(t)u_{0}\|_{H^{s}} < K exp(- \gamma (t-T_{1})), \hspace{3mm} \forall \hspace{1mm} t > T_{1},
\end{align*} 
where $K=K(\|f\|_{H^{1}},\gamma)$ and $T_{1}=T_{1}(\|f\|_{H^{1}},\|u_{0}\|_{H^{s}}, \gamma).$
\end{theorem}
In Theorem \ref{intro theorem}, the map $L_{1}$ is uniformaly compact and $L_{2}$ uniformly convergs to $0$ in $H^{s}$. Therefore, from  \cite[Theorem $1.1.1$]{R}, we get the existence of global attractor. For the proof of Theorem \ref{intro theorem}, we consider the following equation:

\begin{align}   \label{intro1}
 &\partial_{t}u + \partial_{x}^{3}u + 6\left(u^{2} - \frac{1}{2\pi}\|u\|^{2}_{L^{2}}\right)\partial_{x}u + \gamma u = F \hspace{4mm} t > 0, x \in \mathbb{T}, \\
 &u(x,0) = u_{0}(x) \label{intro2}
\end{align}
where 
$$F = f \left(x+ \int\limits_{0}^{t} \|u(\tau)\|_{L^{2}}^{2}d\tau \right).$$
If we put $q(x,t) = u(x+ \int\limits_{0}^{t} \|u(\tau)\|_{L^{2}}^{2}d\tau, \hspace{.5mm} t),$ then $q$ satisfies Equations (\ref{intro1})-(\ref{intro2}).

 \par
 
We divide this paper into six sections. In Section $2$, we describe the preliminaries required for the present paper. Section $3$ descirbes the proof of trilinear estimate by using the Strichartz estimate for mKdV equation proved by J. Bourgain \cite{B}. Section $4$ contains \textit{a priori} estimates. We describe the proof of Theorem \ref{intro theorem} in Section $5.$ Finally in Section $6$, some multilinear estimates are proved. 


\section{Preliminaries}
In this section, we present the notations and definitions which are used throughout this article.
\subsection{Notations}
In this subsection, we list the notations which we use throughout this paper. $C,c$ are the various time independent constants which depend on $s$ unless specified. $a+$ and $a-$ represent $a+\epsilon$ and $a-\epsilon$, respectively for arbitrary small $\epsilon > 0.\hspace{1.5mm} A \lesssim B$ denotes the estimate of the form $A \leq CB$. Similarly, $A \sim B$ denotes $A \lesssim B$ and $B \gtrsim A.$ \par
Define $(dk)_{\lambda}$ to be normalized counting measure on $\mathbb{Z}/\lambda$:
$$\int \phi(k) (dk)_{\lambda} = \frac{1}{\lambda} \sum\limits_{k \in \mathbb{Z}/\lambda} \phi(k).$$
Let $\hat{f}(k)$ and $\tilde{f}(k,\tau)$ denotes the Fourier transform of $f(x,t)$ in $x$ and in $x$ and $t$, respectively.
We define the Sobolev space $H^{s}([0,\lambda])$ with the norm 
$$\|f\|_{H^{s}} = \|\hat{f}(k)\langle k \rangle^{s}\|_{L^{2}((dk)_{\lambda})},$$
where $\langle \cdot \rangle = (1 + |\cdot|).$ For details see \cite{CKSTT02},\cite{T}. We define the space $X^{s,b}$ embedded with the norm
$$\|u\|_{X^{s,b}} = \| \langle k \rangle^{s} \langle \tau - 4\pi^{2}k^{3} \rangle \tilde{u}(k,\tau)\|_{L^{2}((dk)_{\lambda}d\tau)}.$$
We often study the KdV and mKdV equation in $X^{s,\frac{1}{2}}$ space but it hardly contorls the norm $L^{\infty}_{t}H^{s}_{x}$ see \cite{B},\cite{CKSTT02},\cite{T}. To ensure the continuity of the solution, we define a slightly smaller space with the norm
$$\|u\|_{Y^{s}} = \|u\|_{X^{s,\frac{1}{2}}} + \|\langle k \rangle^{s}\tilde{u}(k,\tau)\|_{L^{2}((dk)_{\lambda})L^{1}(d\tau)}.$$
$Z^{s}$ space is defined via the norm
$$\|u\|_{Z^{s}} = \|u\|_{X^{s,-\frac{1}{2}}} + \| \langle k \rangle^{s} \langle \tau - 4\pi^{2}k^{3} \rangle^{-1} \tilde{u}(k,\tau)\|_{L^{2}((dk)_{\lambda})L^{1}(d\tau)}.$$
For the time interval $[t_{1},t_{2}],$ we define the restricted spaces $X^{s,b}$ and $Y^{s}$ embedded with the norms 
\begin{align*}
\|u\|_{X^{s,b}_{([0,\lambda] \times [t_{1},t_{2}])}} &= \inf \lbrace \|U\|_{X^{s,b}}  : U|_{([0,\lambda] \times [t_{1},t_{2}])} = u \rbrace, \\
\|u\|_{Y^{s}_{([0,\lambda] \times [t_{1},t_{2}])}} &= \inf \lbrace \|U\|_{Y^{s}}  : U|_{([0,\lambda] \times [t_{1},t_{2}])} = u \rbrace.
\end{align*}
We state the mean value theorem as follow:
\begin{lemma}
If $a$ is controlled by $b$ and $|k_{1}| \ll |k_{2}|,$ then 
$$a(k_{1} + k_{2}) - a(k_{2}) = O\left(|k_{1}|\frac{b(k_{2})}{|k_{2}|} \right).$$
\end{lemma}
For details see \cite[Section 4]{CKSTT02}.
\subsection{Rescaling}
In this subsection, we rescale the mKdV equation. We can rewrite equations (\ref{intro1})-(\ref{intro2}) in $\lambda$-rescaled form as follow:
\begin{align}    \label{rescaled1}
&\partial_{t} v + \partial_{xxx} v + 6\left(v^{2} - \frac{1}{2\pi}\|v\|_{L^{2}}^{2} \right)\partial_{x}v + \lambda^{-3} \gamma v = \lambda^{-3} g, \\
&v(x,t_{0}) = v_{t_{0}}(x),   \label{rescaled2}
\end{align}
where 
\begin{align*}
g(x,t) &= \lambda^{-1}F(\lambda^{-1}x,\lambda^{-3}t), \\
v_{t_{0}}(x) &= \lambda^{-1}u(\lambda^{-1}x,\lambda^{-3}t_{0}),
\end{align*} 
for initial time $t_{0}.$ If $u$ is the solution of the equations (\ref{intro1})-(\ref{intro2}), then $v(x,t) = \lambda^{-1}u(\lambda^{-1}x,\lambda^{-3}t)$ is the solution of the equations (\ref{rescaled1})-(\ref{rescaled2}). Rescaling is helpful in proving the local in time result as well as \textit{a priori} estimate.
\subsection{I-Operator}
We define an operator $I$ which plays an important role for the $I$-method. Let $\phi : \mathbb{R} \rightarrow \mathbb{R}$ be a smooth monotone $\mathbb{R}$-valued function defined as:
\[ \phi(k) = 
\begin{cases}
1  &|k| < 1, \\
|k|^{s-1}  &|k| > 2. 
\end{cases}
\]
Then, for $m(k) = \phi(\frac{k}{N}),$ we define
\[ m(k) = 
\begin{cases}
1  &|k| < N, \\
|k|^{s-1}N^{1-s}  &|k| > 2N, 
\end{cases}
\]
where we fix $N$ to be a large cut-off. We define the operator $I$ as:
$$\widehat{Iu}(k) = m(k)\hat{u}(k).$$
We can rescale the operator $I$ as follow:
$$\widehat{I'u}(k) = m'(k)\hat{u}(k),$$
where $m'(\frac{k}{\lambda}) = m(k).$ Let $N' = \frac{N}{\lambda}.$ Then
\[ m'(k) = 
\begin{cases}
1  &|k| < N', \\
|k|^{s-1}N'^{(1-s)}  &|k| > 2N'. 
\end{cases}
\]
We use the rescaled $I$-operator for proving the local results for mKdV equation in time. Moreover, proving \textit{a priori} estimate also use the same operator.
\subsection{Strichartz Estimate}
Strichartz estimate plays an important role for the proof of the trilinear estimate. Bourgain in \cite{B}, proves the $L^{4}$ Strichartz estimate for mKdV equation. In the present article, we use the same estimate.
We list the following result:
\begin{prop}   \label{st1}
Let $b > \frac{1}{3}.$ Then, we have
$$\|u\|_{L^{4}(\mathbb{R} \times \mathbb{T})} \lesssim C \|u\|_{X^{0,b}}.$$
\end{prop}

%
\subsection{Local-Wellposedness}
In this subsection, we state the local result in time which can be proved by using the contraction mapping. Let $\eta(t) \in C_{0}^{\infty}$ be a cut-off function such that:
\begin{equation*}   \label{lw:4}
\eta(t) =
\begin{cases}
1 & \text{if}\ |t| \leq 1, \\
0 & \text{if}\ |t| >2. 
\end{cases}
\end{equation*}
Suppose that
 $$D_{\lambda}(t)f(x) =\int e^{2i\pi k x}e^{-(2i\pi k)^{3}t} \hat{f}(k)(dk)_{\lambda}.$$
We assume the following well known lemmas:
\begin{lemma}   \label{lw : lemma1}
$$\| \eta (t) D_{\lambda}(t)w\|_{X^{1,\frac{1}{2}}} \leq \|w\|_{H^{1}}.$$
\end{lemma}

\begin{lemma}   \label{lw : lemma2}
Let $F \in {X^{1,-\frac{1}{2}}}.$ Then  
$$\| \eta (t)\int_{0}^{t} D_{\lambda}(t-t')F(t')dt'\|_{Y^{1}} \leq \|F\|_{Z^{1}}.$$
\end{lemma}
For the proof of Lemmas \ref{lw : lemma1} and \ref{lw : lemma2} see \cite{CKSTT02}.
\begin{prop}    \label{local wellposdeness}
Let $\frac{1}{2} \leq s < 1.$ Then the IVP (\ref{rescaled1})-(\ref{rescaled2}) is locally well-posed for the initial data $v_{t_{0}}$ satisfying $I'v_{t_{0}} \in \dot H^{1}(\mathbb{T})$ and $I'g \in \dot H^{1}(\mathbb{T}).$ Moreover, there exists a unique solution on the time interval $[t_{0},t_{0} + \delta]$ with the lifespan $\delta \sim (\|I'v_{t_{0}}\|_{H^{1}} + \lambda^{-3}\|I'g\|_{H^{1}} + \gamma \lambda^{-3})^{-\alpha}$ for some $\alpha > 0$ and the solution satisfies
\begin{align*}
\|I'v\|_{Y^{1}} &\lesssim \|I'v_{t_{0}}\|_{H^{1}} + \lambda^{-3}\|I'g\|_{H^{1}}, \\
\sup\limits_{t_{0} \leq t \leq t_{0} + \delta}^{}\|I'v(t)\|_{H^{1}} &\lesssim \|I'v_{t_{0}}\|_{H^{1}} + \lambda^{-3}\|I'g\|_{H^{1}}. 
\end{align*}  
\end{prop}
\begin{rem}
Note that 
\begin{align*}
g(x,t) &= \lambda^{-1}F(\lambda^{-1}x,\lambda^{-3}t) \\
&= \lambda^{-1}f\left(x + \frac{1}{2\pi}\int\limits_{0}^{t} \|I'v\|_{L^{2}}^{2} \right)
\end{align*}
\end{rem}
\begin{proof}

The proof of the Proposition \ref{local wellposdeness} follows along the same lines as for KdV equation given in \cite{T} with the help of trilinear estimate given in Proposition \ref{TL for L^{1}}. The only difference arises in the estimate of $g$ as it depends on unknown $u$. To deal with this issue, we define a new metric. Indeed, let 
$$B= \lbrace w\in X^{1,\frac{1}{2}} : \|w\|_{X^{1,\frac{1}{2}}} \lesssim C\left( \|I'v_{0}\|_{H^{1}} + \lambda^{-3}\|I'g\|_{H^{1}} \right)	 \rbrace$$
and define the metric 
$$d(w,w') = \|w-w'\|_{X^{0,\frac{1}{2}}} + \|v-v'\|_{X^{0,\frac{1}{2}}},$$
for $I'v = w.$ As ${X^{0,\frac{1}{2}}}$ is reflexive, the ball $B$ is complete with respect to the metric $d$ for details see \cite[9.14 and Lemma 7.3]{Kato} . Therefore, it is enough to show 
\begin{align*}
\|N(v,w) - N(v',w')\|_{Y^{0}} &\lesssim \|\eta(t)(P(v,w)-P(v',w'))\|_{Z^{0}} \\
&\lesssim \left(\gamma \lambda^{-3} + \lambda^{0+}\left( \|I'v_{0}\|_{H^{1}} + \lambda^{-3}\|I'g\|_{H^{1}} \right)^{2} +\lambda^{-3}\|I'g\|_{H^{1}} \right) \\
&\left( \|w-w'\|_{X^{0,\frac{1}{2}}} + \|v-v'\|_{X^{0,\frac{1}{2}}} \right),
\end{align*}
where 
$$N(w)= \eta(t)D_{\lambda}(t)I'v_{0} - \eta(t)\int D_{\lambda}(t-t')\eta(t')P(t')dt'$$
with
$$P(v,w) = 6I'\left(v^{2} - \frac{1}{2\pi}\|v\|^{2}_{L^{2}}\right)\partial_{x}v + \gamma \lambda^{-3}w - \lambda^{-3}I'g.$$

As the metric consist of both $w$ and $u$ terms, we consider the pair of equation as:
\begin{align}
&\partial_{t} v + \partial_{xxx} v + 6\left(v^{2} - \frac{1}{2\pi}\|v\|_{L^{2}}^{2} \right)\partial_{x}v + \lambda^{-3} \gamma v = \lambda^{-3} g, \label{local equation 2} \\
&\partial_{t} w + \partial_{xxx} w + 6I'\left(v^{2} - \frac{1}{2\pi}\|v\|_{L^{2}}^{2} \right)\partial_{x}(I')^{-1}w + \lambda^{-3} \gamma w = \lambda^{-3} I'g.  \label{Local equation 1}
\end{align}
The estimate of $v$ in $H^{s}$ follows from that of $w$ in $H^{1}$ because $\|v\|_{H^{s}} \lesssim \|w\|_{H^{1}} .$ Therefore, we do not need to assume extra condition on ball for the variable $``v"$ .Let
\begin{align*}
g'(x,t) &= \lambda^{-1}F(\lambda^{-1}x,\lambda^{-3}t) \\
&= \lambda^{-1}f\left(x + \frac{1}{2\pi}\int\limits_{0}^{t} \|I'v'\|_{L^{2}}^{2} \right)
\end{align*}
At first, we consider the external forcing term for Equation \eqref{local equation 2} as:

\begin{equation*}
\begin{split}
&\|I'g - I'g'\|_{X^{0,-\frac{1}{2}}} \lesssim \|I'g-I'g'\|_{L^{2}} \\
=&\Biggl|\Biggl|\lambda^{-1}I'f\left(\lambda^{-1}x + \int_{0}^{\lambda^{-3}t}\|\lambda v(\lambda \cdot,\lambda^{3}\tau )\|_{L^{2}}^{2}d\tau \right) -\\
& \lambda^{-1}I'f\left(\lambda^{-1}x + \int_{0}^{\lambda^{-3}t}\|\lambda v'(\lambda \cdot,\lambda^{3}\tau )\|_{L^{2}}^{2}d\tau \right) \Biggl|\Biggl|_{L^{2}} \\
&\lesssim \left\|\lambda^{-1}\int_{0}^{1}\frac{d}{d\theta}I'f(\lambda^{-1}x + \theta\alpha(t) + (1-\theta)\beta(t))d\theta\right\|_{L^{2}}
\end{split}
\end{equation*}
where
$$\alpha(t) = \int_{0}^{\lambda^{-3}t}\|\lambda v(\lambda \cdot,\lambda^{3}\tau )\|_{L^{2}}^{2}d\tau \hspace{6mm}\text{and}\hspace{6mm} \beta(t) = \int_{0}^{\lambda^{-3}t}\|\lambda v'(\lambda \cdot,\lambda^{3}\tau )\|_{L^{2}}^{2}d\tau.$$
Now from mean value theorem and the fact that translation is invariant, we get
\begin{align*}
\|I'g-I'g'\|_{L^{2}} \lesssim \|I'g\|_{H^{1}}\|v-v'\|_{X^{0,\frac{1}{2}}}.
\end{align*}
Similarly for Equation \eqref{Local equation 1}, we get
\begin{align*}
\|g-g'\|_{L^{2}} \lesssim \|g\|_{H^{1}}\|v-v'\|_{X^{0,\frac{1}{2}}}.
\end{align*}
The nonlinear term can be estimated similar to the $4$-linear estimate of Lemma \ref{Energy 1 estimate}. Note that the $4$-linear estimate has third order derivative on the other hand the nonlinear term has only one. We can make the similar cases for the nonlinear term as given in Integrals $(1)-(3)$ and prove the estimate. Hence, we can use the contraction principle. This shows that the solution $u \in X^{1,\frac{1}{2}}.$ We need to show that the solution belongs to $Y^{1}.$ But from Proposition \ref{TL for L^{1}}, the nonlinear term of the integral equation belongs to $Y^{1}.$ In the same way, we can verify other two terms of integral equation by using Schwarz inequality. Therefore, the solution $u \in Y^{1}.
$
\end{proof}

\section{Trilinear Estimate}
Define an operator $J$ such that 
\begin{align}   \label{operator J}
\hat{J}[u,v,w] = i\frac{k}{3} \sum\limits_{\substack{k_{1} +k_{2} + k_{3} = k \\(k_{1} + k_{2})(k_{2} +k_{3})(k_{3} + k_{1})\neq 0 }}\hat{u}(k_{1})\hat{v}(k_{2})\hat{w}(k_{3}) -i k\hat{u}(k)\hat{v}(k)\hat{w}(-k).
\end{align}
where $\hat{u} \hspace{1.5mm} \text{and} \hspace{1.5mm} \tilde{v}$ denote the Fourier transforms in $x$ variable and both $x \hspace{1mm} \text{and}\hspace{1mm} t$ variables, respectively. We establish the following trilinear estimate for $J$:
\begin{prop}     \label{TL Main result}
Let $s \geq \frac{1}{2} \hspace{1mm}\text{and}\hspace{1mm}u,v,w \in {X^{s,\frac{1}{2}}}$ are $\lambda$-periodic in $x$ variable. Then, we have 
\begin{equation} \label{TL1}
\|J[u,v,w]\|_{X^{s,-\frac{1}{2}}} \leq C\lambda^{0+}\|u\|_{X^{s,\frac{1}{2}}}\|v\|_{X^{s,\frac{1}{2}}}\|w\|_{X^{s,\frac{1}{2}}}.
\end{equation}
\end{prop}

\begin{rem}
We note that if $u$ is real valued, then 
\begin{equation} \label{TL2}
J[u,u,u] = \left(u^{2} - \frac{1}{2\pi}\|u\|^{2}_{L^{2}}\right)\partial_{x}u.
\end{equation}
yields the nonlinearity of mKdV. The first term and the second term of (\ref{operator J}) can be estimated in $H^{s}$ for $s\geq \frac{1}{4}$ and $s \geq \frac{1}{2}$, respectively. So, the bound $s=\frac{1}{2}$ comes from the second term. 
\end{rem}
Simple computations yield
\begin{align*}
\left(u^{2} - \frac{1}{2\pi}\|u\|^{2}_{L^{2}}\right)\partial_{x}u =& i \sum\limits_{\substack{k_{1} +k_{2} + k_{3} = k \\(k_{1} + k_{2})\neq 0 }} \hat{u}(k_{1})\hat{u}(k_{2})k_{3} \hat{u}(k_{3}) \\
=& i \lbrace  \sum\limits_{\substack{k_{1} +k_{2} + k_{3} = k \\(k_{1} + k_{2})(k_{2} +k_{3})(k_{3} + k_{1})\neq 0 }}\hat{u}(k_{1})\hat{u}(k_{2}) k_{3}\hat{u}(k_{3}) \\
&+ \sum\limits_{\substack{k_{1} +k_{2} + k_{3} = k \\(k_{1} + k_{2})(k_{3} + k_{1})\neq 0 \\ (k_{2} +k_{3}) = 0 }}\hat{u}(k_{1})\hat{u}(-k_{3}) k _{3}\hat{u}(k_{3}) \\
&+ \sum\limits_{\substack{k_{1} +k_{2} + k_{3} = k \\(k_{1} + k_{2})(k_{2} +k_{3})\neq 0 \\(k_{3} + k_{1}) =0 }}\hat{u}(-k_{3})\hat{u}(k_{2}) k_{3}\hat{u}(k_{3})  \\
&+ \sum\limits_{\substack{k_{1} +k_{2} + k_{3} = k \\(k_{1} + k_{2})\neq 0 \\ (k_{2} +k_{3}) = (k_{3} + k_{1}) = 0}} k_{3} \hat{u}(k_{1})\hat{u}(-k_{3})^{2} \rbrace \\
=& i\frac{k}{3} \{ \sum\limits_{\substack{k_{1} +k_{2} + k_{3} = k \\(k_{1} + k_{2})\neq 0 }} \hat{u}(k_{1})\hat{u}(k_{2}) \hat{u}(k_{3}) \} \\
&- ik|\hat{u}(k)|^{2} \hat{u}(k).
\end{align*}
\begin{rem}
Note that the right hand side of the above formula is equivalent to $\hat{J}.$ Therefore, the nonlinearity of mKdV equation can be control if we prove the Proposition \ref{TL Main result}
.
\end{rem}
\begin{rem}
If $u$ is a complex-valued function, then we have only to consider 
$$\left(|u^{2}| - \frac{1}{2\pi}\|u\|^{2}_{L^{2}}\right)\partial_{x}u - \frac{i}{2\pi}Im\langle \partial_{x}u, u \rangle_{L^{2}}u$$
instead of the left hand side of the above equality. This yield the nonlinearity of the complex mKdV. 
\end{rem}
\begin{proof} [Proof of Proposition \ref{TL Main result}]
We first consider the trilinear estimate corresponding to non resonant frequencies. We claim that 
$$\left\| i \frac{k}{3} \int\limits_{\substack{k_{1} +k_{2} + k_{3} = k \\(k_{1} + k_{2})(k_{2} +k_{3})(k_{3} + k_{1})\neq 0 }}\hat{u}_{1}(k_{1})\hat{u}_{2}(k_{2})\hat{u}_{3}(k_{3})\right\|_{X^{s,-\frac{1}{2}}} \lesssim \prod\limits_{i=1}^{3} \|u_{i}\|_{X^{s,\frac{1}{2}}}.$$
From duality, it is enough to show
\begin{equation}    \label{TL *}
\left| \int\limits_{\substack{k_{1} +k_{2} + k_{3} + k_{4} = 0 \\(k_{1} + k_{2})(k_{2} +k_{3})(k_{3} + k_{1})\neq 0 }} \langle k_{1}\rangle\int\limits^{}_{\sum\limits_{i=1}^{4}\tau_{i} = 0}\prod\limits_{i=1}^{4} \tilde{u}_{i}(k_{i},\tau_{i})(dk_{i})_{\lambda}d\tau_{i}\right|\lesssim \prod\limits_{i=1}^{3} \|u_{i}\|_{X^{s,\frac{1}{2}}} \|u_{4}\|_{X^{-s,\frac{1}{2}}}.
\end{equation}
Consider LHS of (\ref{TL *}) and let the region of the first integration to be $``*"$ and region of the second integration is denoted by $``**"$. Define $\sigma_{i} = \tau_{i} - 4\pi k^{3}_{i} \hspace{1.5mm} \text{for}\hspace{1.5mm}1 \leqq i \leq 4.$ Multiply and divide by $\langle k_{4} \rangle^{\frac{1}{2}} \langle \sigma_{4}\rangle^{\frac{1}{2}}$ to get
\begin{equation}    \label{TL3}
\left| \int\limits_{*} \int\limits^{}_{**} \langle k_{1} \rangle^{1} \langle k_{4}\rangle^{s} \langle \sigma_{4} \rangle^{-\frac{1}{2}} \tilde{u}_{1} \tilde{u}_{2} \tilde{u}_{3} (\langle k_{4} \rangle^{-s} \langle \sigma_{4} \rangle^{\frac{1}{2}} \tilde{u}_{4})\right|.
\end{equation}
We divide this estimate into following four cases:
\begin{enumerate}
\item Let $| \sigma_{4} | = \max\{| \sigma_{i}| \hspace{1mm}\text{for} \hspace{1mm} 1 \leq i \leq 4 \}.$
\item Let $| \sigma_{3} | = \max\{| \sigma_{i}| \hspace{1mm}\text{for} \hspace{1mm} 1 \leq i \leq 4 \}.$
\item Let $| \sigma_{2} | = \max\{| \sigma_{i}| \hspace{1mm}\text{for} \hspace{1mm} 1 \leq i \leq 4 \}.$
\item Let $| \sigma_{1} | = \max\{| \sigma_{i}| \hspace{1mm}\text{for} \hspace{1mm} 1 \leq i \leq 4 \}.$
\end{enumerate}
From the symmetry and the duality argument, it is enough to show for Case $1$ because other cases can be treated in the same way. As we know, $k_{1} + k_{2} + k_{3} + k_{4} = 0$ and $\tau_{1} + \tau_{2} + \tau_{3} + \tau_{4} =0,$ from simple calculations, we have 
\begin{equation}   \label{TL4}
\langle \sigma_{4} \rangle \gtrsim 3(|k_{1} + k_{2}||k_{2} + k_{3}||k_{3} + k_{1}|) \sim 3(|k_{2} + k_{3}||k_{3} + k_{4}||k_{4} + k_{2}|).
\end{equation}
From symmetry, we can assume that $|k_{1}| \geq |k_{2}| \geq |k_{3}|.$ Now we can again subdivide all three cases into four cases:
\begin{itemize}
\item[1a] $|k_{1}| \sim |k_{2}| \sim |k_{3}| \sim |k_{4}|$
\item[1b] $|k_{1}| \sim |k_{4}| \gg |k_{2}| \gtrsim |k_{3}|$
\item[1c] $|k_{1}| \sim |k_{4}| \sim |k_{2}| \gtrsim |k_{3}|$
\end{itemize} 

\begin{rem}
Note that there are other cases also but if we consider $|k_{1}|\gg |k_{4}|$, the derivative corresponding to $|k_{4}|$ get very small and the estimate is easy to verify.
\end{rem}

\begin{lemma}     \label{TL all are equal}
For \textbf{Case $1a$}, we give the following proof:
\end{lemma}
\begin{proof}
\renewcommand{\qedsymbol}{}
Note that we wish to prove
\begin{align}      \label{TL M}
\|\partial_{x}M(u,u,u)\|_{X^{s,-\frac{1}{2}}} \lesssim \|u\|^{3}_{X^{s,\frac{1}{2}}},
\end{align}
where 
$$\mathcal{F}_{x}[M(u,v,w)] = \sum\limits_{\substack{k_{1} +k_{2} + k_{3} = k \\|k_{1}| \sim |k_{2}| \sim |k_{3}| }}\hat{u}(k_{1}) \hat{v}(k_{2}) \hat{w}(k_{3}),$$
and $\mathcal{F}$ denotes the Fourier transform in $x$ variable. Hence,
\begin{align*}
\|\partial_{x}M(u,u,u)\|_{X^{s,-\frac{1}{2}}} \sim & \left( \int\limits_{k}^{} \langle k \rangle^{3} \left( \int\limits_{-\infty}^{\infty} \langle \sigma \rangle^{-1} \left| \mathcal{F}_{x,t}[M(u,u,u)] \right|^{2} d\tau \right) (dk)_{\lambda} \right)^{\frac{1}{2}} \\
\sim & \| (\langle k \rangle^{\frac{1}{2}}|\tilde{u}|)^{3} \langle \sigma \rangle^{-\frac{1}{2}} \|_{L^{2}(\mathbb{T} \times \mathbb{R})},
\end{align*} 
where $\mathcal{F}_{x,t}$ is the Fourier transform in both $x$ and $t$ variables. Let $\tilde{v}(k,\tau) = \langle k \rangle^{\frac{1}{2}} |\tilde{u}(k,\tau)|.$ Hence, we get 
\begin{align*}
\| (\langle k \rangle^{\frac{1}{2}}|\tilde{u}|)^{3} \langle \sigma \rangle^{-\frac{1}{2}} \|_{L^{2}(\mathbb{T} \times \mathbb{R})} &\lesssim \|v^{3}\|_{X^{0,-\frac{1}{2}}}, \\
&\lesssim \|v^{3}\|_{L^{\frac{4}{3}}(\mathbb{T} \times \mathbb{R})}, 
\end{align*}
From the duality of Strichartz's estimate and Proposition \ref{st1}, we get
\begin{align*}
\| (\langle k \rangle^{\frac{1}{2}}|\tilde{u}|)^{3} \langle \sigma \rangle^{-\frac{1}{2}} \|_{L^{2}(\mathbb{T} \times \mathbb{R})} &\lesssim \|v\|^{3}_{L^{4}(\mathbb{T} \times \mathbb{R})}, \\
&\lesssim \lambda^{0+}\|u\|^{3}_{_{X^{s,\frac{1}{2}}}}.
\end{align*}
 Therefore, we can handle Case $1a$ directly. 
\end{proof}

\textbf{Case $1b.$}   We assume that the size of the Fourier support of $u_j$ satisfies 
\begin{align}
& |k_1|\sim|k_4| \gg |k_2|, |k_3|, \nonumber\\
& |\sigma_4| \gtrsim |k_2 + k_3||k_3 + k_4||k_4 + k_2|,\nonumber \\
& \frac{1}{\lambda} \leq |k_2 + k_3| \leq 1. \label{periodic 6}
\end{align} 

\textbf{Remark 1.} The restriction $k_1 + k_2 + k_3 + k_4 = 0$ and the assumption imply that $|k_1| \sim |k_4|.$ But it does not follow that $|k_2| \sim |k_3|$ unless (\ref{periodic 6}) additionally assumed.\\

We prove the following estimate of the quardlinear functional on $\mathbb{R} \times \lambda\mathbb{T}$ with parameter $\lambda \geq 1$.

\begin{lemma}
For the above conditions, we have
\begin{align}
& \left| \int\limits_{*} \int\limits^{}_{**} \langle k_{1} \rangle^{1} \langle k_{4}\rangle^{s} \langle \sigma_{4} \rangle^{-\frac{1}{2}} \tilde{u}_{1} \tilde{u}_{2} \tilde{u}_{3} (\langle k_{4} \rangle^{-\frac{1}{2}} \langle \sigma_{4} \rangle^{\frac{1}{2}} \tilde{u}_{4})\right| \nonumber \\
 \lesssim & (1 + \lambda^{0+}) \min \lbrace \|u_2\|_{X^{1/4+,1/2}} \|u_3\|_{X^{0,1/2}} , \|u_2\|_{X^{0,1/2}}\|u_3\|_{X^{1/4+,1/2}} \rbrace \times \|u_1\|_{X^{s,1/2}} \|u_4\|_{X^{-s,1/2}}. \label{periodic 1}
\end{align}
\end{lemma}
\renewcommand{\qedsymbol}{}
\begin{proof}
We follow the argument in \cite[Case 3 in the proof of Proposition 5 on page 733-734]{CKSTT02}. We first note that
\begin{align} \label{periodic 2}
|\sigma_4| \gtrsim |k_2+k_3||k_1|^{2}.
\end{align}
From the Plancherel theorem, inequality  (\ref{periodic 2}) and the Sobolev embedding, the left side of (\ref{periodic 1}) can be bounded by the following inequalities.
\begin{align}
&\left| \int\limits_{*} \int\limits^{}_{**} \langle k_{1} \rangle^{1} \langle k_{4}\rangle^{s} \langle \sigma_{4} \rangle^{-\frac{1}{2}} \tilde{u}_{1} \tilde{u}_{2} \tilde{u}_{3} (\langle k_{4} \rangle^{-s} \langle \sigma_{4} \rangle^{\frac{1}{2}} \tilde{u}_{4})\right| \nonumber\\
& \lesssim  \int\limits_{*} \int\limits^{}_{**} \langle k_{1} \rangle^{s} |\bar{\tilde{u}}_1 (k_1)| (|k_2 + k_3|^{-1/2} |\tilde{u}_2 (k_2)| |\tilde{u}_3 (k_3)|)|\sigma_4|^{1/2} |k_{4}|^{-s} |\tilde{u}_4 (k_4)|  d\tau \nonumber \\
& \lesssim \|D_{x}^{s}v_1\|_{L^{4}(\mathbb{R} \times  \lambda \mathbb{T})} \|D_{x}^{-1/2} (v_2 v_3)\|_{L^{4}(\mathbb{R} \times  \lambda \mathbb{T})} \|v_4\|_{X^{-s,1/2}}\nonumber \\
& \lesssim \|v_1\|_{X^{s,1/3+}} \|D_{x}^{-1/4} (v_2 v_3)\|_{L^{4}(\mathbb{R};L^{2}(\lambda \mathbb{T}))} \|v_4\|_{X^{-s,1/2}},  \label{periodic 3}
\end{align} 
where  $\tilde{v_j} = |\tilde{u_j}|$. Furthermore, by the Plancherel's theorem, $1/\lambda \leq |k_{2}| + |k_{3}| \leq 1,$ Schwarz inequality and the Young's inequality, we have
\begin{align}
&\|D_{x}^{-1/4} (v_2 v_3)\|_{L^{2}(\lambda\mathbb{T})} \lesssim \int\limits_{1/ \lambda \leq |k_{23}| \leq 1} |k_{23}|^{-1/2} \left|\hspace{1.5mm} \int\limits_{k_{23} = k_{2} + k_{3}} \tilde{v_{2}}(k_{2})\tilde{v_{3}}(k_{3}) \right|^{2} \nonumber\\
&\lesssim \Bigg( \int\limits_{1/ \lambda \leq |k_{23}| \leq 1} |k_{23}|^{-1} \Bigg)^{1/2} \Bigg(\int\limits_{1/ \lambda \leq |k_{23}| \leq 1} \Bigg|\hspace{1.5mm} \int\limits_{k_{23} = k_{2} + k_{3}} \tilde{v_{2}}(k_{2})\tilde{v_{3}}(k_{3}) \Bigg|^{4}\Bigg)^{1/2} \nonumber \\
&\lesssim (1 + \log\lambda)^{1/2} \min \lbrace \|v_2\|_{L^{2}(\lambda\mathbb{T})}^{2} \|v_3\|_{H^{1/4 +}(\lambda\mathbb{T})}^{2} , \|v_3\|_{L^{2}(\lambda\mathbb{T})}^{2} \|v_2\|_{H^{1/4 +}(\lambda\mathbb{T})}^{2} \rbrace. \label{periodic 4}
\end{align}
The integration in $t$ over $\mathbb{R}$ of the squared left side of (\ref{periodic 4}) yield
\begin{align}
\|D_{x}^{-1/4} (v_2 v_3)\|_{L^{4}(\mathbb{R} ;L^{2}(\lambda\mathbb{T}))} \nonumber 
\lesssim  (1 + \lambda^{0+}) \min \lbrace \|v_2\|_{L^{8}(\mathbb{R};L^{2}(\lambda\mathbb{T}))}^{2}\\ \|v_3\|_{L^{8}(\mathbb{R};H^{1/4 +}(\lambda\mathbb{T}))}^{2} ,  \|v_3\|_{L^{8}(\mathbb{R};L^{2}(\lambda\mathbb{T}))}^{2} \|v_2\|_{L^{8}(\mathbb{R};H^{1/4 +}(\lambda\mathbb{T})}^{2} \rbrace
 \nonumber \\
 \lesssim (1 + \lambda^{0+}) \min \lbrace \|v_2\|^{2}_{X^{0,1/2}} \|v_3\|^{2}_{X^{1/4 +,1/2}}, \|v_2\|^{2}_{X^{0,1/2}} \|v_3\|^{2}_{X^{1/4 +,1/2}} \rbrace. \label{periodic 5}
\end{align}
Accordingly, from (\ref{periodic 3})-(\ref{periodic 5}) we obtained the desire inequality (\ref{periodic 1}).
\end{proof}

	\textbf{Case 1c.} Inequality (\ref{periodic 2}) becomes
	$$|\sigma_{4}| \gtrsim |k_{2} + k_{4}||k_{1}|^{2}.$$
	Therefore, we can estimate case \textbf{1c} in the similar way as case \textbf{1b}. \\
	
\textbf{For the resonant part} (the second term of operator $J$ (\ref{operator J})), the proof is similar to Lemma \ref{TL all are equal} with $M$ defined in the formula (\ref{TL M}) changes to the following:
$$\mathcal{F}_{x}[M(u,u,u)] = |\hat{u}(k)|^{2} |\hat{u}(k)| .$$
\end{proof}
Now, we prove the trilinear estimate corresponding to the function space $Z^{s}$:
\begin{prop}     \label{TL for L^{1}}
For $s \geq \frac{1}{2} \hspace{1mm}\text{and}\hspace{1mm}u,v,w \in {X^{s,\frac{1}{2}}}$, we have 
\begin{equation} \label{TL5}
\|J[u,v,w]\|_{Z^{s}} \leq C \lambda^{0+} \|u\|_{Y^{s}}\|v\|_{Y^{s}}\|w\|_{Y^{s}}.
\end{equation}
\end{prop}
\begin{proof}
From Proposition \ref{TL Main result}, it is enough to show 
$$\|\langle k \rangle^{s} \langle k \rangle \langle \sigma \rangle^{-1} J[u,v,w]\|_{L^{2}_{(dk)_{k}}L^{1}_{d\tau}} \leq C\|u\|_{X^{s,\frac{1}{2}}}\|v\|_{X^{s,\frac{1}{2}}}\|w\|_{X^{s,\frac{1}{2}}}.
$$
Similar to Proposition \ref{TL Main result}, we also divide this problem into the following four cases. 
\begin{enumerate}
\item Let $|\sigma| = \max\{|\sigma|,|\sigma_{i}| \hspace{1mm}\text{for} \hspace{1mm} 1 \leq i \leq 3 \}.$
\item Let $|\sigma_{1}| = \max\{|\sigma|,|\sigma_{i}| \hspace{1mm}\text{for} \hspace{1mm} 1 \leq i \leq 3 \}.$
\item Let $|\sigma_{2}| = \max\{|\sigma|,|\sigma_{i}| \hspace{1mm}\text{for} \hspace{1mm} 1 \leq i \leq 3 \}.$
\item Let $|\sigma_{3}| = \max\{|\sigma|,|\sigma_{i}| \hspace{1mm}\text{for} \hspace{1mm} 1 \leq i \leq 3 \}.$
\end{enumerate}
Case $1$ is the worst one. Indeed, otherwise we have by Schwarz's inequality,
\begin{align*}
&\|\langle k \rangle^{s} \langle k \rangle \langle \sigma \rangle^{-1} \sum\limits_{k} \hat{u}_{1} \hat{u}_{2} \hat{u}_{3} \|_{L^{2}_{(dk)_{k}}L^{1}_{\tau}} \\
&\lesssim \left\|\left( \int\limits_{-\infty}^{\infty} \frac{1}{\langle \sigma \rangle^{2(\frac{1}{2} + \epsilon)}} d\tau \right)^{\frac{1}{2}} \left( \int\limits_{-\infty}^{\infty} \frac{\langle k \rangle^{2s} \langle k \rangle^{2}}{\langle \sigma \rangle^{2(\frac{1}{2} - \epsilon)}} \left| \sum\limits_{k}^{} \hat{u}_{1} \hat{u}_{2} \hat{u}_{3} \right|^{2} d\tau \right)^{\frac{1}{2}} \right\|_{L^{2}_{(dk)_{\lambda}}}. \\
&\lesssim C \left\| \frac{\langle k \rangle \langle k \rangle^{s}}{\langle \sigma \rangle^{(\frac{1}{2} - \epsilon)}} \sum\limits_{k}^{} \hat{u}_{1} \hat{u}_{2} \hat{u}_{3} d\tau \right\|_{L^{2}_{(dk)_{k}}L^{2}_{\tau}},
\end{align*}
and hence it reduces to the same proof as in Proposition \ref{TL Main result}. Therefore, we only have to prove Case $1.$ From symmetry, assume that $|k_{1}| \geq |k_{2}| \geq |k_{3}|.$ We divide Case $1$ into further three cases as follow: 
\begin{itemize}
\item [1a.] $|k_{1}| \sim |k_{2}| \sim |k_{3}|.$
\item [1b.] $|k_{1}| \gg |k_{2}| \gtrsim |k_{3}|.$
\item [1c.] $|k_{1}| \sim |k_{2}| \gg |k_{3}|.$
\end{itemize} 
\textbf{Case $1a$}. By the Schwarz's inequality, we have
\begin{align*}
& \int\limits_{-\infty}^{\infty} \langle \sigma \rangle^{-1} |\mathcal{F}_{t,x}[M(u,u,u)]|d\tau \\
&\leq \left(\int\limits_{-\infty}^{\infty} \langle \sigma \rangle^{-1 - \epsilon}d\tau \right)^{\frac{1}{2}} \left(\int\limits_{-\infty}^{\infty} \langle \sigma \rangle^{-1 + \epsilon}|\mathcal{F}_{t,x}[M(u,u,u)]|^{2} d\tau \right)^{\frac{1}{2}},
\end{align*}
where $M$ is defined in (\ref{TL M}). This case is reduces to Lemma \ref{TL all are equal}. \\
\textbf{Case $1b$}. In this case, we can clearly see that $\langle \sigma \rangle \gtrsim |k_{2} + |k_{3}||(\langle k \rangle^{2} + \langle \sigma \rangle)$. Due to symmetry, we can assume that $|k| \sim |k_{1}|.$ By using Schwarz's inequality, we get 
\begin{align*}
&\left|\sum\limits_{k} \langle k \rangle^{s} \langle k \rangle \langle \sigma \rangle^{-1}  \hat{u}_{1} \hat{u}_{2} \hat{u}_{3} \right|_{L^{2}_{(dk)_{k}}L^{1}_{\tau}} \\
&\lesssim \left\|\sum\limits_{k}^{}\left( \int\limits_{-\infty}^{\infty} \frac{\langle k \rangle^{2}}{\langle \sigma \rangle^{2}} d\tau \right)^{\frac{1}{2}} \left( \int\limits_{-\infty}^{\infty} \langle k \rangle^{2s} \left|  \hat{u}_{1} \hat{u}_{2} \hat{u}_{3} \right|^{2} d\tau \right)^{\frac{1}{2}} \right\|_{L^{2}_{(dk)_{\lambda}}}.
\end{align*}
As we can see 
\begin{align*}
\left( \int\limits_{-\infty}^{\infty} \frac{\langle k \rangle^{2}}{\langle \sigma \rangle^{2}} d\tau \right)^{\frac{1}{2}} &\lesssim \left( \int\limits_{-\infty}^{\infty} \frac{\langle k \rangle^{2}}{(\langle \sigma \rangle + |k_{2} + k_{3}| \langle k \rangle^{2})^{2}} d\tau \right)^{\frac{1}{2}}, \\
&= \left( \int\limits_{-\infty}^{\infty} \frac{\langle k \rangle^{2}}{(|\tau - k^{3}| + |k_{1} + k_{2}|\langle k \rangle^{2})^{2}} d\tau \right)^{\frac{1}{2}}, \\
&= \left( \int\limits_{-\infty}^{k^{3}} \frac{\langle k \rangle^{2}}{( k^{3}-\tau +|k_{2} + k_{3}| \langle k \rangle^{2})} d\tau \right)^{\frac{1}{2}} + \left( \int\limits_{k^{3}}^{\infty} \frac{\langle k \rangle^{2}}{(\tau - k^{3} + |k_{2} + k_{3}|\langle k \rangle^{2})} d\tau \right)^{\frac{1}{2}} \\
& \lesssim C|k_{2} + k_{3}|^{-1/2}.
\end{align*}
Hence, from H\"older's inequality, Proposition \ref{st1} and inequality (\ref{periodic 5}), we get
\begin{align*}
&\left\|\sum\limits_{k}^{} |k_{2} + k_{3}|^{-1/2} \langle k \rangle^{s}  \hat{u}_{1} \hat{u}_{2} \hat{u}_{3} \right\|_{L^{2}_{(dk)_{\lambda}} L^{2}_{\tau}} \\
 &\sim \left\|    \sum\limits_{k}^{} (|k_{1}|^{s}\hat{u}_{1}) (|k_{2} + k_{3}|^{-1/2}\hat{u}_{2} \hat{u}_{3}) \right\|_{L^{2}_{(dk)_{\lambda}} L^{2}_{\tau}}, \\
&\lesssim \|D_{x}^{s}u_{1}\|_{L^{4}_{x,t}} \|D_{x}^{-\frac{1}{2}}(u_{2}u_{3})\|_{L^{4}_{x,t}}  \\
&\lesssim \lambda^{0+} \|u_{1}\|_{X^{s,\frac{1}{3}+}} \|u_{2}\|_{X^{\frac{1}{4}+,\frac{1}{2}}} \|u_{3}\|_{X^{0,\frac{1}{2}}}.
\end{align*}
The estimate for the resonant term follows in the same way as Case $1a$.
\end{proof}
Let $u = u_{L} + u_{H}$ where $supp \hspace{1mm}\hat{u}_{L}(k) \subset \{|k| \ll N\}$ and $supp \hspace{1mm}\hat{u}_{H}(k) \subset \{|k| \gtrsim N\}.$ We prove the following corollary:

\begin{cor}    \label{TL corollary}
Let $1 \gg \epsilon\geq 0.$ Let $u,v,w \in X^{s,\frac{1}{2} - \epsilon}.$ Then, the following three estimates hold:
\begin{enumerate}
\item If $v,u$ are low and $w$ is high frequency functions, then we have
\begin{align*}
&\left\|(u_{L}v_{L} - \sum\limits_{l = -\infty}^{\infty}\hat{u}_{L}(l)\hat{v}_{L}(-l))w_{H}\right\|_{X^{1-2\epsilon,-\frac{1}{2} + \epsilon}} \\
&\lesssim\lambda^{0+} C\min\{ \|u_{L}\|_{X^{\frac{1}{2} + \epsilon,\frac{1}{2} - \epsilon}} \|v_{L}\|_{X^{0 ,\frac{1}{2} - \epsilon}}, \|v_{L}\|_{X^{\frac{1}{2} + \epsilon,\frac{1}{2} - \epsilon}} \|u_{L}\|_{X^{0 ,\frac{1}{2} - \epsilon}}\}\|w_{H}\|_{X^{0,\frac{1}{2} - \frac{\epsilon}{2}}}.
\end{align*}
\item If $v,w$ are high and $u$ is low frequency functions, then 
\begin{align*}
&\left\|(u_{L}v_{H} - \sum\limits_{l = -\infty}^{\infty}\hat{u}_{L}(l)\hat{v}_{H}(-l))w_{H}\right\|_{X^{1-2\epsilon,-\frac{1}{2} + \epsilon}} \\
&\lesssim\lambda^{0+} C\min\{ \|u_{L}\|_{X^{\frac{1}{2} + \epsilon,\frac{1}{2} - \epsilon}} \|v_{H}\|_{X^{0 ,\frac{1}{2} - \epsilon}}, \|v_{H}\|_{X^{\frac{1}{2} + \epsilon,\frac{1}{2} - \epsilon}} \|u_{L}\|_{X^{0 ,\frac{1}{2} - \epsilon}}\}\|w_{H}\|_{X^{0,\frac{1}{2} - \frac{\epsilon}{2}}}.
\end{align*}
\item If $u,v\hspace{1.2mm}\text{and}\hspace{1.2mm} w$ all are high frequency functions, then
\begin{align*}
&\left\|(u_{H}v_{H} - \sum\limits_{l = -\infty}^{\infty}\hat{u}_{H}(l)\hat{v}_{H}(-l))w_{H}\right\|_{X^{-2\epsilon,\frac{1}{2} + \epsilon}} \\
& \lesssim \lambda^{0+} \|u_{H}\|_{X^{0,\frac{7}{18} + \epsilon}} \|v_{H}\|_{X^{0,\frac{7}{18} + \epsilon}} \|w_{H}\|_{X^{0,\frac{7}{18} + \epsilon}}.
\end{align*}
\end{enumerate}
\end{cor}
\begin{proof}
\textbf{1.} We know that 
$$\mathcal{F}_{x}\left[ (u_{L}v_{L} - \sum\limits_{l = -\infty}^{\infty}\hat{u}_{L}(l)\hat{v}_{L}(-l))w_{H} \right] = \sum\limits_{\substack{k_{1} +k_{2} + k_{3} +k_{4} = 0 \\ k_{1} + k_{2} \neq 0 \\ (k_{1} + k_{2})(k_{2} + k_{3})(k_{3} + k_{1}) \neq 0}} \hat{u}_{L}(k_{1})\hat{v}_{L}(k_{2})\hat{w}_{H}(k_{3}),$$
where $\mathcal{F}_{x}$ denotes the Fourier transform in the $x$ variable. Hence, we need to show that 
\begin{align*}
&\left\| \sum\limits_{k}^{} e^{ikx} \int\limits_{\substack{k_{1} +k_{2} + k_{3} = k \\(k_{1} + k_{2})(k_{2} +k_{3})(k_{3} + k_{1})\neq 0 }} \langle k_{1} \rangle^{1-2\epsilon} \hat{u}_{L}(k_{1}) \hat{v}_{L}(k_{2}) \hat{w}_{H}(k_{3}) \right\|_{X^{0,-\frac{1}{2} + \epsilon}} \\
&\lesssim C\min\{ \|u_{L}\|_{X^{\frac{1}{2} + \epsilon,\frac{1}{2} - \epsilon}} \|v_{L}\|_{X^{0 ,\frac{1}{2} - \epsilon}}, \|v_{L}\|_{X^{\frac{1}{2} + \epsilon,\frac{1}{2} - \epsilon}} \|u_{L}\|_{X^{0 ,\frac{1}{2} - \epsilon}}\}\|w_{H}\|_{X^{0 ,\frac{1}{2} - \frac{\epsilon}{2}}}.
\end{align*}
From duality, it is enough to show
\begin{align} \label{TL6}
&\left| \int\limits_{\substack{k_{1} +k_{2} + k_{3} = k \\(k_{1} + k_{2})(k_{2} +k_{3})(k_{3} + k_{1})\neq 0 }} \int\limits_{\sum\limits^{4}_{i=1}\tau_{i} = 0}^{} \langle k_{4} \rangle^{1 -2\epsilon} \tilde{u}_{1}(k_{1}) \tilde{u}_{2}(k_{2}) \tilde{u}_{3}(k_{3}) \tilde{u}_{4}(k_{4}) \right| \\
&\lesssim C\min\{ \|u_{L}\|_{X^{\frac{1}{2} + \epsilon,\frac{1}{2} - \epsilon}} \|v_{L}\|_{X^{0 ,\frac{1}{2} - \epsilon}}, \|v_{L}\|_{X^{\frac{1}{2} + \epsilon,\frac{1}{2} - \epsilon}} \|u_{L}\|_{X^{0 ,\frac{1}{2} - \epsilon}}\}\|w_{H}\|_{X^{0 ,\frac{1}{2} - \frac{\epsilon}{2}}} \nonumber.
\end{align}
where $u_{1} = u_{L},\hspace{1mm} u_{2} = v_{L},\hspace{1mm} u_{3} = w_{H}$ and let $u_{4} = u_{L} + u_{H}.$ Let $\sigma_{i} = \tau_{i} - 4\pi^{2}k^{3}_{i} \hspace{1.5mm} \text{for} \hspace{1.5mm} 1 \leq i \leq  4.$ We divide the proof into the following four cases:
\begin{enumerate}
\item Let $|\sigma_{4}| = \max\{|\sigma_{i}| \hspace{1mm}\text{for} \hspace{1mm} 1 \leq i \leq 4 \}.$
\item Let $|\sigma_{1}| = \max\{|\sigma_{i}| \hspace{1mm}\text{for} \hspace{1mm} 1 \leq i \leq 4 \}.$
\item Let $|\sigma_{2}| = \max\{|\sigma_{i}| \hspace{1mm}\text{for} \hspace{1mm} 1 \leq i \leq 4 \}.$
\item Let $|\sigma_{3}| = \max\{|\sigma_{i}| \hspace{1mm}\text{for} \hspace{1mm} 1 \leq i \leq 4 \}.$
\end{enumerate}
It is enough to prove for Case $1$ because other cases can be treated in the same way. According to the given conditions, we have $|k_{1}|, |k_{2}| \ll N'$ and $|k_{3}| \sim |k_{4}| \gtrsim N'.$ So, from (\ref{TL4}), $\langle \sigma_{4} \rangle \gtrsim \langle k_{4} \rangle ^{2}|k_{3} + |k_{4}| \hspace{1.5mm} \text{and} \hspace{1.5mm} 1/\lambda \leq |k_{3} + k_{4}| \leq 1.$ Let the region for the first integration is denoted as $``*"$ and the region of second integration is denoted as $``**".$ By using Plancherel's theorem, H\"older's inequality, for the term (\ref{TL6}), we get
\begin{align*}
&\left| \int\limits_{* } \int\limits_{**}^{} \langle k_{4}\rangle^{1-2\epsilon} \tilde{u}_{1} \tilde{u}_{2} \tilde{u}_{3} \tilde{u}_{4} \right| \\
&\lesssim \left| \int\limits_{* } \int\limits_{**}^{} \langle k_{4} \rangle^{1-2\epsilon} \langle k_{4} \rangle^{-1+2\epsilon} (|k_{1} + k_{2}|^{-1/2}|\tilde{u}_{1}| |\tilde{u}_{2}|) |\tilde{u}_{3}| (|\tilde{u}_{4}|\langle \sigma_{4} \rangle^{\frac{1}{2} - 2\epsilon}) \right|, \\
&\lesssim \|D_{x}^{-1/2}(v_{1}v_{2})\|_{L^{4}_{x,t}} \|v_{3}\|_{L^{4}_{x,t}} \|\tilde{v}_{4}\langle \sigma_{4} \rangle^{\frac{1}{2} - 2\epsilon} \|_{L^{2}_{k,\tau}}.
\end{align*}
for $v_{j} = |u_{j}|.$ From Sobolev embedding, inequality (\ref{periodic 5}) and Proposition \ref{st1}, we get the desired inequality. \\
\textbf{2.} We can prove this case along the similar line. \\
\textbf{3.} Form duality argument and Proposition \ref{st1}, we get the desire estimate.
\end{proof}
\begin{lemma} \label{Infinity Estimate}
\begin{equation}
\|u\|_{L^{\infty}_{x,t}} \lesssim \|u\|_{X^{\frac{1}{2} + \epsilon, \frac{1}{2} + \epsilon}}.
\end{equation}
\end{lemma}
\begin{proof}
$$\|u\|_{L^{\infty}_{t}L^{2}_{x}}^{2} = \sup\limits_{t \in \mathbb{R}}^{}\|U(-t)u(t)\|_{L^{2}_{x}}^{2},$$
where $U(t) = e^{-t\partial_{x}^{3}}.$ By Sobolev embedding, we have
\begin{align*}
\sup\limits_{t \in \mathbb{R}}^{}\|U(-t)u(t)\|_{L^{2}_{x}}^{2} &\lesssim \int \sup\limits_{t \in \mathbb{R}}^{}|U(-t)u(t)|^{2}dx \\
&\lesssim \int \langle \partial_{t} \rangle^{\frac{1}{2} + \epsilon} |U(-t)u(t)|^{2}dx \\
&\sim \|u\|^{2}_{X^{0,\frac{1}{2} +\epsilon}}.
\end{align*}
Hence, we get
$$\|u\|_{L^{\infty}_{t}L^{2}_{x}}^{2} \lesssim \|u\|^{2}_{X^{0,\frac{1}{2} +\epsilon}}.$$
\end{proof}

\section{A Priori Estimate}
In this section, we show a priori estimate of the solution to the mKdV equation which are needed for the proof of Theorem \ref{intro theorem}. The energy for the mKdV equation is given as:
\begin{equation} \label{Priori energy}
E(u) = \int (\partial_{x}u)^{2}- (u)^{4} dx.
\end{equation}
For the operator $I',$ we have
$$E(I'v) = \int (\partial_{x}I'v)^{2}- (I'v)^{4} dx.$$
From equations (\ref{rescaled1})-(\ref{rescaled2}), we obtain
\begin{align}   \label{Energy}
\dv{(E(I'v))}{t} =& \left[\int (-\partial_{x}^{2} I'v - (I'v)^{3})(-\partial_{x}^{3} I'v - \partial_{x} I'v^{3})\right] \nonumber \\
&+ \left[\int -\lambda^{-3} \partial_{x}^{2} I'v I'g - \lambda^{-3} (I'v)^{3} I'g + \dfrac{1}{2}(I'v)^{4} \gamma \lambda^{-3}\right].
\end{align}
For a Banach space $X,$ we define the space $L^{\infty}_{T'}X$ via the norm:
$$\|u\|_{L^{\infty}_{T'}X} = \sup_{t \in [0,T']} \|u(t)\|_{X}.$$

Multiply equation (\ref{rescaled1}) by $v$ and take $L^{2}$ norm to obtain the following lemma:
\begin{lemma}   \label{L2 bound}
$$\|v(t)\|_{L^{2}}^{2} \lesssim \|v_{0}\|_{L^{2}}\hspace{1mm} exp(-\gamma \lambda^{-3}t) + \frac{\lambda^{-3}}{\gamma}\|g\|_{L^{\infty}_{t}L^{2}}^{2}(1- exp(-\gamma \lambda^{-3} t)).$$
\end{lemma}
We establish the following lemma:
\begin{lemma}   \label{Energy 2 estimate}
Let $v$ is the solution of IVP (\ref{rescaled1})-(\ref{rescaled2}) for $t\in [0,T'].$ Then, we have 
\begin{align}   \label{eq11}
\|I'v(T')\|_{L^{2}}^{2} exp(\gamma \lambda^{-3}T') \leq C_{1}(\|v(0)\|_{L^{2}}^{2} + \frac{1}{\gamma}\|g\|_{L^{2}}^{2} exp(\gamma \lambda^{-3}T'))
\end{align}
and
\begin{align}  \label{eq12}
\|I'v(T')\|^{2}_{\dot H^{1}} exp(\gamma \lambda^{-3} T') \leq C_{1} \Big(\|I'v(0)\|^{2}_{\dot H^{1}} + \frac{1}{\gamma^{2}} \|I'g\|^{2}_{L^{\infty}_{T'}\dot H^{1}} exp(\gamma \lambda^{-3} T') \nonumber \\ 
+ \|v(0)\|_{L^{2}}^{6} + \frac{1}{\gamma^{4}} \|g\|^{6}_{L^{2}}exp(\gamma \lambda^{-3}T')\Big) +\left| \int \limits_{0}^{T'} M(t)dt\right|,
\end{align}
where 
$$M(t) =exp(\gamma \lambda^{-3} t) \int_{\lambda\mathbb{T}} \lbrace -\partial_{x}^{2} I'v - (I'v)^{3}\rbrace \lbrace-\partial_{x} I'v^{3}-\partial_{x}^{3} I'v \rbrace.$$
\end{lemma}
\begin{proof}
Similar to Lemma \ref{L2 bound}, we have 
\begin{align*}
\frac{d}{dt}\|v(T')\|_{L^{2}}^{2} exp(\gamma \lambda^{-3}T') &=  \left( -\gamma\lambda^{-3}\|v(t)\|_{L^{2}} + 2\lambda^{-3} \int_{\lambda\mathbb{T}}v(t)g(t)dx \right)exp(\gamma\lambda^{-3}T') \\
&\leq \frac{\lambda^{-3}}{\gamma}\|g\|^{2}_{L^{2}}exp(\gamma\lambda^{-3}T').
\end{align*}
Intriguing over $[0,T']$ and from the definition of operator $I$, we get \eqref{eq11}. \par 
From equations (\ref{rescaled1})-(\ref{rescaled2}), we get 
\begin{align*}
&\dv{}{t} \left( E(I'v(t))exp(\gamma \lambda^{-3} t') \right) \hspace{4.5cm} \\  
=& \dv{}{t} E(I'v(t))exp(\gamma \lambda^{-3} t') + \gamma \lambda^{-3} E(I'v(t))exp(\gamma \lambda^{-3} t'), \\
=& \left[ \int \lbrace- \partial_{x}^{2} I'v - (I'v)^{3} \rbrace \lbrace \lambda^{-3}I'g - \gamma \lambda^{-3}I'v - \partial_{x}^{3} I'v - \partial_{x} I'v^{3} \rbrace \right]exp(\gamma \lambda^{-3} t') \\
&+ \gamma \lambda^{-3}  exp(\gamma \lambda^{-3} t') \int \frac{1}{2} (\partial_{x} I'v)^{2} - \dfrac{1}{4}(I'v)^{4}, \\
=& \left[\int (-\partial_{x}^{2} I'v - (I'v)^{3})(-\partial_{x}^{3} I'v - \partial_{x} I'v^{3})\right]exp(\gamma \lambda^{-3} t') \\
&+ \left[\int (-\partial_{x}^{2} I'v - (I'v)^{3})({\lambda}^{-3} I'g - \gamma \lambda^{-3} I'v)\right]exp(\gamma \lambda^{-3} t') \\
&+ \gamma \lambda^{-3}  exp(\gamma \lambda^{-3} t')\int \dfrac{1}{2} (\partial_{x} I'v)^{2} - \dfrac{1}{4}(I'v)^{4},  \\ 
=& M(t')+  \left[\int -\lambda^{-3} \partial_{x}^{2} I'v I'g - \lambda^{-3} (I'v)^{3} I'g -\frac{1}{2}\gamma\lambda^{-3}(\partial_{x}I'v)^{2}+ \dfrac{3}{4}(I'v)^{4} \gamma \lambda^{-3}\right] exp(\gamma \lambda^{-3} t'). 
\end{align*}
 Put the value of $E,$ integrate over $[0,T']$, take absolute value on both side and from Gagliardo-Nirenberg inequality, we get
\begin{align*}
&\left(\|I'v(T')\|_{\dot H^{1}}^{2} - \|I'v(T')\|_{L^{4}}^{4} \right)exp(\gamma\lambda^{-3}T')\\
=& \|I'v(0)\|_{\dot H^{1}}^{2} - \|I'v(0)\|_{L^{4}}^{4} +\int\limits_{0}^{T'} M(t')dt' +\int\limits_{0}^{T'} \Big[\int -\lambda^{-3} \partial_{x}^{2} I'v I'g - \lambda^{-3} (I'v)^{3} I'g \\
&-\frac{1}{2}\gamma\lambda^{-3}(\partial_{x}I'v)^{2}+ \dfrac{3}{4}(I'v)^{4}  \gamma \lambda^{-3}\Big] exp(\gamma \lambda^{-3} t')dt', \\
\lesssim &\|I'v(0)\|_{\dot H^{1}}^{2} - \|I'v(0)\|_{L^{4}}^{4}+ \left| \int\limits_{0}^{T'} M(t')dt'\right| + \lambda^{-3}\int\limits_{0}^{T'}\Big[\|I'g\|_{\dot H^{1}}\|I'v(t')\|_{\dot H^{1}} \\
&+ \|I'v(t')\|_{\dot H^{1}}\|I'v(t')\|^{2}_{L^{2}}\|I'g\|_{L^{2}}  
-\gamma\frac{1}{2}\|I'v(t')\|_{\dot H^{1}}^{2} \\
&+\gamma \frac{3}{4}  \|I'v(t')\|_{\dot H^{1}} \|I'v(t')\|^{3}_{L^{2}} \Big] exp(\gamma \lambda^{-3} t')dt'. 
\end{align*}

From Young's inequality, we have
\begin{align*}
&\|I'v(T')\|_{\dot H^{1}}^{2}exp(\gamma\lambda^{-3}T') \lesssim \|I'v(0)\|_{\dot H^{1}}^{2} + \frac{1}{\gamma^{2}}\|I'g\|_{L^{\infty}_{T'}\dot H^{1}}^{2}exp(\gamma\lambda^{-3}T') \\
+& \left| \int\limits_{0}^{T'} M(t')dt'\right| 
+ C_{1}\|I'v(T')\|^{6}_{L^{2}}exp(\gamma\lambda^{-3}T') \\
&+C_{1} \int\limits_{0}^{T'} \left(\|I'v(t')\|_{L^{2}}^{6} + \frac{1}{\gamma^{2}}\|I'v(t')\|_{L^{2}}^{4}\|I'g\|_{L^{2}}^{2}\right)\gamma\lambda^{-3} exp(\gamma\lambda^{-3}t')dt'.
\end{align*}
From inequality \eqref{eq11} we get
\begin{align*}
\left(\|I'v(t')\|_{L^{2}}^{6} + \frac{1}{\gamma^{2}}\|I'v(t')\|_{L^{2}}^{4}\|I'g\|_{L^{2}}^{2}\right) \lesssim \|I'v(0)\|_{L^{2}}^{6}exp(-3\gamma\lambda^{-3}t') + \frac{1}{\gamma^{3}}\|I'g\|^{6}_{L^{2}}. 
\end{align*}
and hence we obtain inequality \eqref{eq12}.
\end{proof}
\begin{rem}
For mKdV equation, we just consider the half part of damping term  in $exp(\gamma\lambda^{-3}T')$ as compare to KdV equation.
\end{rem}
%
%

We need to state the following Leibnitz rule type lemma:
\begin{lemma}     \label{proposition 0}
$$\|f(t)g(x,t)\|_{X^{s,b}} \lesssim \|\hat{f}\|_{L^{1}}\|g\|_{X^{s,b}} + \|f\|_{H^{b}_{t}}\|\langle k \rangle^{s} \tilde{g}\|_{L^{2}_{(dk)_{\lambda}}L^{1}_{d\tau}}.$$
\end{lemma}
\begin{proof}
Assume that $\tau = \tau_{1} + \tau_{2}.$ Let $\sigma = \tau - k^{3}, \sigma_{1} = \tau_{1}$ and $\sigma_{2} = \tau_{2} - k^{3}.$ Then
\begin{align*}
\langle \sigma\rangle^{b} = \langle \tau - k^{3} \rangle^{b} \lesssim \langle \tau_{1} \rangle^{b} + \langle \tau - \tau_{1} - k^{3} \rangle^{b}.
\end{align*}
Hence
\begin{align*}
&\langle \sigma \rangle^{b} \langle k \rangle^{s} \mathcal{F}[f(t)g(x,t)] = \langle \sigma \rangle^{b} \langle k \rangle^{s} \int_{\tau_{1}} \hat{f}(\tau_{1})\tilde{g}(k, \tau - \tau_{1}) d\tau_{1} , \\
&\lesssim \langle k \rangle^{s} \int_{\tau_{1}} \langle \tau_{1} \rangle^{b} |\hat{f}(\tau_{1})\tilde{g}(k, \tau - \tau_{1})| + \langle \tau -\tau_{1} - k^{3} \rangle^{b} |\hat{f}(\tau_{1})\tilde{g}(k, \tau - \tau_{1})| d\tau_{1}.
\end{align*}
After summing over $k$ and taking $L^{2}$ norm, we get 
$$\|\langle \sigma \rangle^{b} \langle k \rangle^{s} \mathcal{F}[f(t)g(x,t)]\|_{L^{2}_{k,\tau}} \leq \|\langle k \rangle^{s} \langle \tau_{1} \rangle^{b}\hat{f} * \tilde{g}\|_{L^{2}_{t,k}} + \|\langle k \rangle^{s} \langle \tau - \tau_{1} - k^{3} \rangle^{b}\hat{f} * \tilde{g}\|_{L^{2}_{t,k}}.$$
From Young's inequality in $\tau$, we obtain 
$$\|\langle k \rangle^{s} \langle \tau_{1} \rangle^{b}\hat{f} * \tilde{g}\|_{L^{2}_{\tau}} + \|\langle k \rangle^{s} \langle \tau - \tau_{1} - k^{3} \rangle^{b}\hat{f} * \tilde{g}\|_{L^{2}_{\tau}} \lesssim \|\hat{f}\|_{L^{1}}\|g\|_{X^{s,b}} + \|f\|_{H^{b}_{t}}\|\langle k \rangle^{s} \tilde{g}\|_{L^{2}_{(dk)_{\lambda}}L^{1}_{d\tau}}. $$
\end{proof}
Similar to \cite[Proposition 3.1]{T}, we finally have the following proposition:
\begin{prop}   \label{proposition 1}
Let $\frac{1}{2} \leq s < 1$. Let $T >0$ is given, $\epsilon >0$ be sufficiently small and $u$ be a solution of IVP (\ref{intro1})-(\ref{intro2}) on $[0,T].$ Assume that $N^{\frac{1}{2}(1 - \epsilon)} \geq \gamma, N^{\epsilon -} \geq C_{6}T$ and 
\begin{align*}
(\|u(0)\|_{L^{2}}^{2} + \frac{1}{\gamma^{2}} \|f\|_{L^{2}}^{2}exp(\gamma T)) \leq N^{\frac{1}{6}(1 - \epsilon)} C_{3} \\
(\|Iu(0)\|_{\dot H^{1}}^{2} + \frac{1}{\gamma^{2}} \|If\|_{\dot H^{1}}^{2}exp(\gamma T)) \leq N^{\frac{1}{6}(1 - \epsilon)} C_{3}.
\end{align*}
Then, we have
\begin{align*}
&\|Iu(T)\|_{L^{2}}^{2}exp(\gamma T) \leq C_{4}(\|u(0)\|_{L^{2}}^{2} + \frac{1}{\gamma^{2}} \|f\|_{L^{2}}^{2}exp(\gamma T)), \\
&\|Iu(T)\|_{\dot H^{1}}^{2}exp(\gamma T) \leq C_{4}(\|Iu(0)\|_{\dot H^{1}}^{2} + \|u(0)\|_{L^{2}}^{6}+ \frac{1}{\gamma^{4}}\|f\|_{L^{2}}^{6}exp(\gamma T)\\
& +  \frac{1}{\gamma^{2}} \|If\|_{\dot H^{1}}^{2}exp(\gamma T)) + (\|Iu(0)\|_{H^{1}}^{2} + \frac{1}{\gamma^{2}} \|If\|_{H^{1}}^{2}exp(\gamma T)),
\end{align*}
where $C_{4}$ is independent of $N$ and $T.$
\end{prop}
\begin{rem}
Without loss of generality, we can replace $f$ with $F$ as $F$ is just a translation of $f.$ 
\end{rem}
We can rescale Proposition \ref{proposition 1} by taking $\lambda = N^{\frac{1}{6}(1-\epsilon)}, N' = \frac{N}{\lambda}, T' = \lambda^{3} T.$ Also, we note that $\|I'v\|^{2}_{\dot H^{1}} = \lambda^{-3}\|Iu\|^{2}_{\dot H^{1}}, \|I'g	\|^{2}_{L^{\infty}_{T'}\dot H^{1}} = \lambda^{-3}\|If\|^{2}_{\dot H^{1}}.$ We rewrite Proposition \ref{proposition 1} as following:
\begin{prop}   \label{proposition 2}
Let $\frac{1}{2} \leq s < 1$, $T' >0$ is given and let $v$ be a solution of IVP (\ref{rescaled1})-(\ref{rescaled2}) on $[0,T'].$ Assume that $\lambda^{3} \geq \gamma$ and that for suitable $C_{6},C_{3} >0,$ $ N'^{-}\lambda^{0-} \geq C_{6}T'\lambda^{2}$ and 
\begin{align*}
(\|v(0)\|_{L^{2}}^{2} + \frac{1}{\gamma^{2}} \|g\|_{L^{2}}^{2}exp(\gamma \lambda^{-3} T')) \leq C_{3} \\
(\|I'v(0)\|_{\dot H^{1}}^{2} + \frac{1}{\gamma^{2}} \|I'g\|_{L^{\infty}_{T'}\dot H^{1}}^{2}exp(\gamma \lambda^{-3} T')) \leq  C_{3}.
\end{align*}
Then, we have	
\begin{align*}
&\|I'v(T')\|_{L^{2}}^{2}exp(\gamma \lambda^{-3} T') \leq C_{4} (\|v(0)\|_{L^{2}}^{2} + \frac{1}{\gamma^{2}} \|g\|_{L^{2}}^{2}exp(\gamma \lambda^{-3} T')) \\
&\|I'v(T')\|_{\dot H^{1}}^{2}exp(\gamma \lambda^{-3} T') \leq C_{4} (\|I'v(0)\|_{\dot H^{1}}^{2} + \frac{1}{\gamma^{2}} \|I'g\|_{L^{\infty}_{T'}\dot H^{1}}^{2}exp(\gamma \lambda^{-3} T') \\
& \hspace{40mm}+ \|v(0)\|_{L^{2}}^{6} + \frac{1}{\gamma^{4}} \|g\|_{L^{\infty}_{T'} L^{2}}^{6}exp(\gamma \lambda^{-3} T')) \\
&\hspace{40mm}+ \lambda^{-2}(\|I'v(0)\|_{H^{1}}^{2} + \frac{1}{\gamma^{2}}\|I'g\|_{H^{1}}^{2}exp(\gamma\lambda^{-3}T')),
\end{align*}
where $C_{4}$ is independent of $N',T'$ and $\lambda.$
\end{prop}
\begin{rem}  \label{sobolev remark}
Because of non homogeneity of non homogeneous Sobolev space, we can not rescale the Proposition \ref{proposition 1} into Proposition \ref{proposition 2} with the order of rescaling factor as $\lambda^{-3}$ like the KdV equation.
Also, if we consider the homogeneous Sobolev space, the trilinear and multilinear estimates may not follows for counterexample see appendix. Therefore, we consider the non homogeneous Sobolev space with the rescaling estimate $\|I'v\|_{H^{1}}^{2} \lesssim \lambda^{-1}\|Iu\|_{H^{1}}^{2}.$  We estimate $L^{2}$ and $\dot H^{1}$ separately to prove Proposition \ref{proposition 2} in $H^{1}$. Although, it is not necessary for our problem to have the separate estimates but for the shake of general proof, we estimate it separately.

\end{rem}
\begin{proof}[Proof of Proposition \ref{proposition 2}]
 
Take $\delta > 0$ and $j \in \mathbb{N}$ such that $\delta j = T'$ where $\delta \sim (\|I'v(0)\|_{H^{1}} + \|I'g\|_{L^{\infty}_{T'}H^{1}} + \gamma \lambda^{-3})^{-\alpha}, \alpha >0.$ For $0 \leq m \leq j,\hspace{1mm} m \in \mathbb{Z},$ we prove 
  \begin{align}
    &\|I'v(m\delta)\|^2_{\dot H^1}exp(\gamma \lambda^{-3}m\delta)\notag \\
    \leq& 2C_1 (\|I'v(0)\|^2_{\dot H^1}+\|v(0)\|^6_{L^2}+\frac{1}{\gamma^2} \|I'g\|^2_{\dot H^1}exp(\gamma \lambda^{-3}m\delta)\notag \\
          & +\frac{1}{\gamma^4}\|g\|^6_{L^2}exp(\gamma \lambda^{-3}k\delta))+ \lambda^{-2}(\|I'v(0)\|_{H^{1}}^{2} + \frac{1}{\gamma^{2}}\|I'g\|_{H^{1}}^{2}exp(\gamma\lambda^{-3}T'))\notag \\
     \leq& 4C_1C_3 + \lambda^{-2}(\|I'v(0)\|_{H^{1}}^{2} + \frac{1}{\gamma^{2}}\|I'g\|_{H^{1}}^{2}exp(\gamma\lambda^{-3}T')) \label{prop prio 2}
  \end{align}
  by induction.

  For $m=0$, \eqref{prop prio 2} hold trivially. We assume \eqref{prop prio 2} hold true for $m=l$ where $0\leq l\leq j-1$. 
From Lemma \ref{Energy 2 estimate}, we have
\begin{align}
 &\|I'v((l+1)\delta)\|^2_{\dot H^1}exp(\gamma \lambda^{-3}(l+1)\delta)\leq  C_1 (\|I'v(0)\|^2_{\dot H^1}+\|v(0)\|^6_{L^2} \nonumber\\ 
& +\frac{1}{\gamma^2} \|I'g\|^2_{\dot H^1}exp(\gamma \lambda^{-3}(l+1)\delta) 
  \  +\frac{1}{\gamma^4}\|g\|^6_{L^2}exp(\gamma \lambda^{-3}(l+1)\delta)+\left|\int^{(l+1)\delta}_0M(t)dt\right| \nonumber
\end{align}
Therefore, it suffices to prove
\begin{align*}
\left| \int \limits_{0}^{(l+1)\delta} M(t)dt \right| +  \lesssim \lambda^{-2}  (\|I'v(0)\|_{H^{1}} + \frac{1}{\gamma^{2}}\|I'g\|_{L^{\infty}_{(l+1)\delta}H^{1}} 
exp(\gamma \lambda^{-3}(l+1)\delta)).
\end{align*}

If $\gamma = 0$ and $f = 0$ in Equation (\ref{Energy}), then we have the following estimate:
\begin{lemma} \label{Energy 1 estimate}
$$\left| \int\limits_{0}^{T'} M(t)dt \right| \lesssim \lambda^{0+} N'^{-1+} \|Iu\|^{4}_{X^{1,\frac{1}{2}}_{T'}} + \lambda^{0+} N'^{-2}\|Iu\|^{6}_{X^{1,\frac{1}{2}}_{T'}}.$$
\end{lemma}
We prove Lemma \ref{Energy 1 estimate} in last section. \\
Lemma \ref{Energy 1 estimate} implies that
\begin{align*}
\left| \int\limits_{0}^{(l+1)\delta} M(t)dt \right| \sim& \sum\limits_{k=0}^{l}\left| \int\limits_{k\delta}^{(k+1)\delta}M(x,t) dt\right|, \\
\lesssim& (N')^{-1+}\lambda^{0+} \sum\limits_{k=0}^{l}\|exp(\frac{1}{4}\gamma \lambda^{-3}t)I'v\|_{X^{1,\frac{1}{2}}_{([0,\lambda] \times [k\delta, (k+1)\delta])}}^{4}\\ 
&+ (N')^{-2}\lambda^{0+} \sum\limits_{k=0}^{l}\|exp(\frac{1}{6}\gamma \lambda^{-3}t)I'v\|_{X^{1,\frac{1}{2}}_{([0,\lambda] \times [k\delta, (k+1)\delta])}}^{6}.
\end{align*}
From Proposition \ref{proposition 0}, we obtain
\begin{align*}
&\left|\int\limits_{0}^{(l+1)\delta} M(t)dt \right| \\
\lesssim & (N')^{-1+}\lambda^{0+} \sum\limits_{k=0}^{l} \|\widehat{exp(\gamma \lambda^{-3}t)}\|_{L^{1}_{[k\delta,(k+1)\delta]}} \|I'v\|^{4}_{X^{1,\frac{1}{2}}_{([0,\lambda] \times [k\delta, (k+1)\delta])}} \\
&+ (N')^{-1+}\lambda^{0+} \sum\limits_{k=0}^{l} \|exp(\gamma \lambda^{-3} t)\|_{H^{\frac{1}{2}}_{[k\delta, (k+1)\delta]}} \|\langle k \rangle^{s} \tilde{I'v}\|^{4}_{L^{2}_{[0,\lambda]}L^{1}_{[k\delta,(k+1)\delta]}} \\
&+ (N')^{-2}\lambda^{0+} \sum\limits_{k=0}^{l} \|\widehat{exp(\gamma \lambda^{-3}t)}\|_{L^{1}_{[k\delta,(k+1)\delta]}} \|I'v\|^{6}_{X^{1,\frac{1}{2}}_{([0,\lambda] \times [k\delta, (k+1)\delta])}} \\
&+ (N')^{-2}\lambda^{0+} \sum\limits_{k=0}^{l} \|exp(\gamma \lambda^{-3} t)\|_{H^{\frac{1}{2}}_{[k\delta, (k+1)\delta]}} \|\langle k \rangle^{s} \tilde{I'v}\|^{6}_{L^{2}_{[0,\lambda]}L^{1}_{[k\delta,(k+1)\delta]}}.
\end{align*}
From simple computations, we can verify that 
$$\max\limits_{0 \leq l \leq k} \|\widehat{exp(\gamma \lambda^{-3}t)}\|_{L^{1}_{[l\delta,(l+1)\delta]}} \lesssim C\hspace{.5mm} exp(\gamma \lambda^{-3}(l+1)\delta)$$ 
and 
 $$\max\limits_{0 \leq l \leq k} \|exp(\gamma \lambda^{-3} t)\|_{H^{\frac{1}{2}}_{[l\delta, (l+1)\delta]}} \lesssim C \hspace{.5mm} exp(\gamma \lambda^{-3}(l+1)\delta) $$ 
	are bounded. 
From the first inequality of Proposition \ref{local wellposdeness}, we have
\begin{align}
\|I'v\|^{4}_{X^{1,\frac{1}{2}}_{([0,\lambda] \times [k\delta, (k+1)\delta])}} + \|\langle \partial_{x} \rangle I'v\|^{4}_{L^{2}_{[0,\lambda]}L^{1}_{[k\delta,(k+1)\delta]}} \lesssim \|I'v(k\delta)\|_{H^{1}_{[0,\lambda]}}^{4} + (\lambda^{-3}\|I'g\|)^{4}_{L^{\infty}_{(l+1)\delta}H^{1}_{[0,\lambda]}}. \label{Priori 5}\\
\|I'v\|^{6}_{X^{1,\frac{1}{2}}_{([0,\lambda] \times [k\delta, (k+1)\delta])}} + \|\langle \partial_{x} \rangle I'v\|^{6}_{L^{2}_{[0,\lambda]}L^{1}_{[k\delta,(k+1)\delta]}} \lesssim \|I'v(k\delta)\|_{H^{1}_{[0,\lambda]}}^{6} + (\lambda^{-3}\|I'g\|)^{6}_{L^{\infty}_{(l+1)\delta}H^{1}_{[0,\lambda]}}. \label{Priori 6}
\end{align}
Therefore, we have
\begin{align}  \label{Lambda 3and4 2}
\left|\int\limits_{0}^{(l+1)\delta}M(t)dt \right| \lesssim (C_{6}\lambda^{2}T')^{-1} \sum\limits_{k=0}^{l}(\|I'v(k\delta)\|_{H^{1}_{[0,\lambda]}}^{4} + (\lambda^{-3}\|I'g\|)^{4}_{L^{\infty}_{(l+1)\delta}H^{1}_{[0,\lambda]}}exp(\gamma\lambda^{-3}(l+1)\delta)).
\end{align}
From inequalities (\ref{Priori 5}),(\ref{Priori 6}) and the assumption in Proposition \ref{proposition 2}, we get 
\begin{align*}
\left|\int\limits_{0}^{(l+1)\delta} M(t)dt\right| \lesssim 2(C_{6}\lambda^{2} T')^{-1} C_{3}(C_{1}^{2} + C_{1}^{3})(l+1)(\|I'v(0)\|_{H^{1}}  \\ 
+\frac{1}{\gamma^{2}}\|I'g\|^{2}_{L^{\infty}_{T'}H^{1}_{[0,\lambda]}}exp(2\gamma \lambda^{-3}(l+1)\delta)).
\end{align*}
We choose $C_{6}$ sufficiently large such that $2(C_{6}T')^{-1} C_{3}(C_{1}^{2} + C_{1}^{3})(l+1) \leq 2(C_{6}\delta)^{-1}C_{3}(C_{1}^{2} + C_{1}^{3}) \ll 1,$ which leads to Proposition \ref{proposition 2}.
\end{proof} 

\section{Proof of Theorem \ref{intro theorem}}
In this section, we describe the proof of Theorem \ref{intro theorem}.
\begin{proof}[Proof of Theorem \ref{intro theorem}]
Let $0 < \epsilon \ll 12s-11$ be fixed. We choose  $T_{1} > 0$ so that 
\begin{align} 
  exp( \gamma T_{1}) > &(\|u_{0}\|^{2}_{H^{s}}+\|u_0\|^6_{L^2})(\frac{1}{\gamma^2} \|f\|^2_{ H^1}+\frac{1}{\gamma^4}\|f\|^6_{L^2})^{-1} \mathop{max} \bigg\lbrace \gamma^{\frac{4(1-s)}{1-\epsilon}}, (C_{6} T_{1})^{\frac{2(1-s)}{\epsilon-}},\notag \\
                       &  \left(\frac{C_{3}}{2}\|u_{0}\|^{-2}_{H^{s}}\right)^{\frac{12(s-1)}{(1-\epsilon)+12(s-1)}} , \left( 2 C_{3}^{-1} \gamma^{-2} \|f\|^{2}_{H^{1}}exp(\gamma T_{1}) \right)^{\frac{6(-2s+2)}{1-\epsilon}} \bigg\rbrace , \label{main11}
\end{align}
which is possible as $\frac{6(-2s+2)}{1-\epsilon} < 1$.  $T_{1}$ depends only on $\|u_{0}\|_{H^{s}}$, $\|f\|_{H^{1}}$ and $\gamma$. 
Set
\begin{align} \label{main12}
N = \mathop{max} \bigg\lbrace \gamma^{\frac{2}{1-\epsilon}},(C_{6} T_{1})^{\frac{1}{\epsilon-}}, \left(  \frac{C_{3}}{2}\|u_{0}\|^{-2}_{H^{s}}\right)^{\frac{-6}{12(1-s)+(1-\epsilon)}} , 
\left( 2 C_{2}^{-1} \gamma^{-2} \|f\|^{2}_{H^{1}} e^{2\gamma T_{1}} \right)^{\frac{6}{1-\epsilon}} \bigg\rbrace.
\end{align}
From the choice of $T_1$ and $N$, we know
$$N^{\frac{1-\epsilon}{2}} \geq \gamma, \ \  \ \ \ \ N^{\epsilon-} \geq C_{6}T_{1},$$
and 
$$\|Iu_{0}\|_{H^{1}}^{2} \leq N^{2-2s} \|u_{0}\|_{H^{s}}^{2}  \leq \dfrac{C_{3}}{2} N^{\frac{1-\epsilon}{6}-},$$
$$\gamma^{-2} \|If\|^{2}_{H^{1}} e^{2\gamma T_{1}}\leq \dfrac{C_{3}}{2} N^{\frac{1-\epsilon}{6}-}.$$
Hence, from Proposition \ref{proposition 1}, we gains
\begin{align}
  \|u(T_{1})\|_{H^{s}}^{2} \leq& \|Iu(T_{1})\|_{H^{1}}^{2}\notag \\
\leq &C_3 (\|Iu_0\|^2_{ H^1}exp(-\gamma T_{1})+\|u_0\|^6_{L^2}exp(-\gamma T_{1})+\frac{1}{\gamma^2} \|If\|^2_{ H^1}+\frac{1}{\gamma^4}\|f\|^6_{L^2})\notag \\
\leq &C_3 (N^{2(1-s)}(\|u_0\|^2_{ H^s}exp(-\gamma T_{1})+\|u_0\|^6_{L^2}exp(-\gamma T_{1}))+\frac{1}{\gamma^2} \|f\|^2_{ H^1}+\frac{1}{\gamma^4}\|f\|^6_{L^2}).\nonumber                                   \end{align}
From \eqref{main11} and \eqref{main12} , we get
$$N^{2(1-s)}exp(-\gamma T_1) (\|u_0\|^2_{ H^s}+\|u_0\|^6_{L^2})< \frac{1}{\gamma^2} \|f\|^2_{ H^1}+\frac{1}{\gamma^4}\|f\|^6_{L^2}$$
which helps us give the bound
$$\|u(T_{1})\|_{H^{s}}^{2} \leq 2C_{3}( \frac{1}{\gamma^2} \|f\|^2_{ H^1}+\frac{1}{\gamma^4}\|f\|^6_{L^2})< K_{1},$$
where $K_{1}$ depends only on $\|f\|_{ H^1}$ and $\gamma$.

In the next place, one can fix $T_{2} > 0$ and solve  mKdV equation on time interval $[T_{1},T_{1} + T_{2}]$ with initial data replaced by $u(T_{1}).$ Let $K_{2} > 0$ be sufficiently large such that

\begin{align} 
 K_{2} exp( \gamma t) > &(\|u_{0}\|^{2}_{H^{s}}+\|u_0\|^6_{L^2})(\frac{1}{\gamma^2} \|f\|^2_{ H^1}+\frac{1}{\gamma^4}\|f\|^6_{L^2})^{-1} \mathop{max} \bigg\lbrace \gamma^{\frac{4(1-s)}{1-\epsilon}}, (C_{6} t)^{\frac{2(1-s)}{\epsilon-}},\notag \\
                       &  \left((C_{3})^{-1}2 K_{1}\right)^{\frac{12(s-1)}{(1-\epsilon)+12(s-1)}} , \left( 2 C_{3}^{-1} \gamma^{-2} \|f\|^{2}_{H^{1}}exp(\gamma T_{1}) \right)^{\frac{6(-2s+2)}{1-\epsilon}} \bigg\rbrace , \label{TH1.2 3}
\end{align}
for any $t > 0$. Set $N^{2(1-s)} = K_{2}exp(\gamma T_{2})$, then inequality \eqref{TH1.2 3} verifies the assumptions in Proposition \ref{proposition 1} and hence we obtain
\begin{align}
\|Iu(T_{1} + T_{2})\|_{H^{1}}^{2} \leq &C_4 (N^{2(1-s)}\|u(T_1)\|^2_{ H^s}exp(\gamma T_2)+\|u(T_1)\|^6_{L^2}exp(-\gamma T_2)+\frac{1}{\gamma^2} \|f\|^2_{ H^1}+\frac{1}{\gamma^4}\|f\|^6_{L^2}) \notag \\
 \leq &C_4 (K_1K_2+K^2_1+\frac{1}{\gamma^2} \|f\|^2_{ H^1}+\frac{1}{\gamma^4}\|f\|^6_{L^2})<K_3. \nonumber
\end{align}

For $t > T_{1},$ we define the maps $L_{1}(t)$ and $L_{2}(t)$ as
$$\widehat{L_{1}(t)u_{0}} = \widehat{S(t)u_{0}}|_{|\zeta| < N_t}, \ \ \  \widehat{L_{2}(t)u_{0}} = \widehat{S(t)u_{0}}|_{|\zeta| > N_t},$$
where $S(t)u_{0} = u(t)$ and $N_t = (K_{2}exp(\gamma(t-T_{1}))^{-\frac{1}{2(1-s)}}.$

It's easy to see that for $t > T_{1},$ 
\begin{align}
\|L_{1}(t)u_{0}\|_{H^{1}}^{2} \leq \|Iu(t)\|_{H^{1}}^{2} &< K_{3}, \notag \\
\|L_{2}(t)u_{0}\|_{H^{s}}^{2} \leq N^{2s-2}\|Iu(t)\|_{H^{1}}^{2} &< K_{2}^{-1} K_{3}ezp(-\gamma(t - T_{1})).\nonumber
\end{align}
 Hence we obtain Theorem \ref{intro theorem} by taking $K = \mathop{max}\lbrace K_{3}^{\frac{1}{2}}, K_{2}^{-\frac{1}{2}}K_{3}^{\frac{1}{2}}\rbrace$.
\end{proof}


	\section{Multilinear Estimates}
In this section, we prove the $4$-linear and $6$-linear estimates given in Lemma \ref{Energy 1 estimate}.
\begin{proof}[Proof of Lemma \ref{Energy 1 estimate}]
For $\gamma =0$ and $g=0$ in (\ref{Energy}), we have
\begin{align*}
\dv{E}{t} (E(I'v)) =& \left[\int (-\partial_{x}^{2} I'v - (I'v)^{3})(-\partial_{x}^{3} I'v - \partial_{x} I'v^{3})\right], \\
E(I'v(T)) - E(I'v(0)) =& \int\limits_{0}^{T}\int\limits_{0}^{\lambda}\partial_{x}^{3}I'v[(I'v)^{3} - I'v^{3}]dx dt + \int\limits_{0}^{T}\int\limits_{0}^{\lambda}\partial_{x}(I'v)^{3}[(I'v)^{3} - I'v^{3}]dx dt, \\
=& I'_{1} + I'_{2},
\end{align*}
for any arbitrary $T > 0.$ For an $\epsilon >0$ let $w_{j} \in X^{s,\frac{1}{2}}$ such that $w|_{[0,\lambda]\times[0,T]} = v_{j}$ and $\|v_{j}\|_{X^{s,\frac{1}{2}}_{T}} \leq C\|w_{j}\|_{X^{s,\frac{1}{2}}} \leq C\|v_{j}\|_{X^{s,\frac{1}{2} + \epsilon}_{T}}$ for $1 \leq j \leq 4.$ Let $\eta_{T}(t) = \eta(t/T)$ and let $\tilde{\eta}$ denotes the Fourier transform only in $t.$ From the Plancherel's theorem, it suffices to prove the following:
\begin{align*}
I'_{1} = \int\limits_{\mathbb{R}}\int\limits_{0}^{\lambda}\eta(t)\partial_{x}^{3}I'w[(I'w)^{3} - I'w^{3}]dx dt, 
\lesssim  \int\limits_{\substack{k_{1} + k_{2} + k_{3} + k_{4} = 0 \\ (k_{1} + k_{2})(k_{2} + k_{3})(k_{3} + k_{1}) \neq 0}} \int\tilde{\eta}(\tau_{1} + \tau_{2} + \tau_{3} + \tau_{4})\\
 \Bigl|\langle k_{1} \rangle^{3}(\widetilde{I'w_{1}})\left( 1- \frac{m(k_{2} + k_{3} +k_{4})}{m(k_{2})m({k_{3}})m({k_{4}})} \right) (\widetilde{I'w_{2}})(\widetilde{I'w_{3}})(\widetilde{I'w_{4}})\Bigl| (dk_{i})_{\lambda}d\tau_{i}   \\
+ \int\limits_{\Omega} \int \Biggl|\langle k_{1} \rangle^{3}(\widetilde{I'w_{1}})\left( 1- \frac{m(k_{2} + k_{3} +k_{4})}{m(k_{2})m({k_{3}})m({k_{4}})} \right)(\widetilde{I'w_{2}})(\widetilde{I'w_{3}})(\widetilde{I'w_{4}})\Bigl| (dk_{i})_{\lambda}d\tau_{i}, 
= I_{11} + I_{12},
\end{align*}
where $\Omega = \lbrace k_{1} + k_{2} + k_{3} + k_{4} = 0 :\hspace{1.5mm} |k_{1} + k_{2}| \neq 0,\hspace{1.5mm}  (|k_{2} + k_{3}||(k_{3} + k_{1}|) = 0 \rbrace$ and $w_{i}  = w_{i}(k_{i},\tau_{i}).$ Let $w = w_{L} + w_{H}$ where $supp \hspace{1mm}\hat{w}_{L}(k) \subset \{|k| \ll N'\}$ and $supp \hspace{1mm}\hat{w}_{H}(k) \subset \{|k| \gtrsim N'\}.$ From dyadic partition of $|k_{i}|,$ we let $|k_{i}| \sim N'_{i}.$ Let $\sigma_{i} = \tau_{i} - 4\pi^{2}k_{i}^{3}$ for $1 \leq i \leq 4.$ We can assume that $\langle \sigma_{4} \rangle = \max \{\langle\sigma_{i} \rangle,\hspace{1mm} 1 \leq i \leq 4 \rbrace$ as all other cases can be treated in the same way. Let $*$ be the region of integration for $I_{11}$. After substituting $w = w_{L} + w_{H},$ we can write $I_{11}$ as a sum of the following three integrals:
\begin{itemize}
\item[Integral 1.]   
\begin{align}
&\int\limits_{*}^{} \int\tilde{\eta}(\tau_{1} + \tau_{2} + \tau_{3} + \tau_{4}) \Bigg|\langle k_{1} \rangle^{3}(\widetilde{I'w_{H}}) 
\nonumber \\ 
&\left( 1- \frac{m(k_{2} + k_{3} +k_{4})}{m(k_{2})m({k_{3}})m({k_{4}})} \right) (\widetilde{I'w_{L}})(\widetilde{I'w_{L}})(\widetilde{I'w_{H}})\Bigg| (dk_{i})_{\lambda}d\tau_{i} .
\end{align}
\item[Integral 2.] 
\begin{align}
&\int\limits_{*}^{} \int\tilde{\eta}(\tau_{1} + \tau_{2} + \tau_{3} + \tau_{4}) \Bigg|\langle k_{1} \rangle^{3}(\widetilde{I'w_{H}} )\nonumber  \\
&\left( 1- \frac{m(k_{2} + k_{3} +k_{4})}{m(k_{2})m({k_{3}})m({k_{4}})} \right) (\widetilde{I'w_{L}})(\widetilde{I'w_{H}})(\widetilde{I'w_{H}})\Bigg| (dk_{i})_{\lambda}d\tau_{i} .
\end{align}
\item[Integral 3.]  
\begin{align}
&\int\limits_{*}^{} \int\tilde{\eta}(\tau_{1} + \tau_{2} + \tau_{3} + \tau_{4}) \Bigg|\langle k_{1} \rangle^{3}(\widetilde{I'w_{H}})\nonumber \\
&\left( 1- \frac{m(k_{2} + k_{3} +k_{4})}{m(k_{2})m({k_{3}})m({k_{4}})} \right) (\widetilde{I'w_{H}})(\widetilde{I'w_{H}})(\widetilde{I'w_{H}})\Bigg| (dk_{i})_{\lambda}d\tau_{i}.
\end{align}
\end{itemize}
\begin{rem}
We omit other cases as they follows in the similar manner.
\end{rem}
\begin{proof}[\textbf{Integral 1.}] \let\qed\relax

For this case, we have $|k_{1}| \sim |k_{4}| \gtrsim N'$ and $|k_{2}| \sim |k_{3}| \ll N'.$ Hence, by using mean value theorem, we get
 \begin{align*}
\left| \left( 1- \frac{m(k_{2} + k_{3} +k_{4})}{m(k_{2})m({k_{3}})m({k_{4}})} \right) \right| \lesssim \frac{|k_{2}| + |k_{3}|}{|k_{4}|} .
 \end{align*}
For \textit{Integral 1}, we get
\begin{align*}
\textit{Integral 1} \lesssim N_{4}^{-1+ 2\epsilon} \int\limits_{*}^{} \int\tilde{\eta}(\tau_{1} + \tau_{2} + \tau_{3} + \tau_{4})(\langle k_{1} \rangle \widetilde{I'w_{H}} \langle \sigma \rangle^{\frac{1}{2} })\bigg[\langle k_{1} \rangle  \lbrace (|k_{2}|\widetilde{I'w_{L}})(\widetilde{I'w_{L}}) +  \\ (\widetilde{I'w_{L}})(|k_{3}|\widetilde{I'w_{L}}) \rbrace(\langle k_{1} \rangle \widetilde{I'w_{H}}) \langle \sigma \rangle^{-\frac{1}{2} }) \bigg].
\end{align*}
Plancherel's theorem, Schwarz's inequality and Corollary \ref{TL corollary}(1) imply
\begin{align*}
\textit{Integral 1} &\lesssim\lambda^{0+} N'^{-1+ 2\epsilon} \|I'w_{H}\|_{X^{1,\frac{1}{2}}} \|I'w_{L}\|_{X^{1,\frac{1}{2}}} (N_{3})^{-\frac{1}{2}}\|I'w_{L}\|_{X^{1,\frac{1}{2}}} \|I'w_{H}\|_{X^{1,\frac{1}{2}}}, \\
&\lesssim\lambda^{0+} N'^{-1+ 2\epsilon} \|I'w\|^{4}_{X^{1,\frac{1}{2}}}.
\end{align*}
Note that, we neglect $(N_{3})^{-\frac{1}{2}}$ as it is not contributing in the decay.
\end{proof}
\begin{proof}[\textbf{Integral 2}] \let\qed\relax

From given conditions, we have $|k_{1}| \sim |k_{4}|\gg |k_{3}| \gtrsim N'$ and $|k_{2}| \ll N'.$ Also, the definition of $m$ implies $m(k_{2}) \sim 1.$ Therefore,
\begin{align*}
\left| \left( 1- \frac{m(k_{2} + k_{3} +k_{4})}{m(k_{2})m({k_{3}})m({k_{4}})} \right) \right| &\lesssim  \frac{m(k_{1})}{m(k_{2})m({k_{3}})m({k_{4}})} \\
&\sim \frac{1}{m(k_{3})} \\
& \lesssim N'^{-1+s}|k_{3}|^{1-s}\\
& \lesssim N'^{-1}|k_{3}|.
\end{align*}
For \textit{Integral 2}, we get
\begin{align*}
&\textit{Integral 2} \\
&\lesssim N'^{-1+ 2\epsilon} \int\limits_{*}^{} \int\tilde{\eta}(\tau_{1} + \tau_{2} + \tau_{3} + \tau_{4})(\langle k_{1} \rangle \widetilde{I'w_{H}} \langle \sigma \rangle^{\frac{1}{2} })\bigg[ \langle k_{1} \rangle ( \widetilde{I'w_{L}})(|k_{3}|	\widetilde{I'w_{H}})(\langle k_{1} \rangle \widetilde{I'w_{H}}) \langle \sigma \rangle^{-\frac{1}{2} }) \bigg].
\end{align*}
From Plancherel's theorem, Schwarz's inequality and Corollary \ref{TL corollary}(2), we have
\begin{align*}
\textit{Integral 2} &\lesssim N'^{-1+ 2\epsilon} N^{-\frac{1}{2}}_{2} \|I'w_{L}\|_{X^{1,\frac{1}{2}}} \|I'w_{H}\|_{X^{1,\frac{1}{2}}} \|I'w_{H}\|_{X^{1,\frac{1}{2}}} \|I'w_{L}\|_{X^{1,\frac{1}{2}}} \\
&\lesssim N'^{-1+2\epsilon} \|I'w\|^{4}_{X^{1,\frac{1}{2}}}.
\end{align*}
\end{proof}
\begin{proof}[\textbf{Integral 3}] \let\qed\relax

Clearly, we have $|k_{1}|\sim |k_{2}|\sim |k_{3}|\sim |k_{4}| \gtrsim N'.$  Hence, from definition of $m,$ we have
\begin{align*}
\left| \left( 1- \frac{m(k_{2} + k_{3} +k_{4})}{m(k_{2})m({k_{3}})m({k_{4}})} \right) \right| &\lesssim  \frac{m(k_{1})}{m(k_{2})m({k_{3}})m({k_{4}})} \\
&\sim \frac{N'^{-2s+2}|k_{1}|^{s-1}}{|k_{2}|^{s-1} |k_{3}|^{s-1} |k_{4}|^{s-1}} |k_{4}| |k_{4}|^{-1} \\
& \lesssim N'^{-2+2s} |k_{2}|^{1-s} |k_{3}|^{1-s} |k_{4}|^{1-s} |k_{1}|^{s-1} |k_{4}||k_{4}|^{-1}\\
& \lesssim N'^{-1}|k_{4}|,
\end{align*}
for $1/2 \leq s <1.$ Therefore, \textit{Integral 3} implies
\begin{align*}
&\textit{Integral 3} \\
&\lesssim N'^{-1+ 2\epsilon} \int\limits_{*}^{} \int\tilde{\eta}(\tau_{1} + \tau_{2} + \tau_{3} + \tau_{4})(\langle k_{1} \rangle \widetilde{I'w_{H}} \langle \sigma \rangle^{\frac{1}{2} })\bigg[ (\langle k_{1} \rangle \widetilde{I'w_{H}})(\langle k_{1} \rangle \widetilde{I'w_{H}})( |k_{4}| \widetilde{I'w_{H}}) \langle \sigma \rangle^{-\frac{1}{2} }) \bigg].
\end{align*}
From Plancherel's theorem, Schwarz's inequality and Corollary \ref{TL corollary}(3), we have
\begin{align*}
\textit{Integral 3} &\lesssim \lambda^{0+} N'^{-1+ 2\epsilon} \|I'w_{H}\|_{X^{1,\frac{7}{18}+}} \|I'w_{H}\|_{X^{1,\frac{7}{18}+}} \|I'w_{H}\|_{X^{1,\frac{7}{18}+}} \|I'w_{H}\|_{X^{1,\frac{7}{18}+}} \\
&\lesssim\lambda^{0+} N'^{-1+2\epsilon} \|I'w\|^{4}_{X^{1,\frac{1}{2}}}.
\end{align*}
\end{proof}
\begin{rem} \label{Symmt}
Note that
\begin{align*}
\left[ k^{3}_{1}\left( 1- \frac{m(k_{2} + k_{3} +k_{4})}{m(k_{2})m({k_{3}})m({k_{4}})} \right) \right]_{sym} = \sum\limits_{j=1}^{4} k_{j}^{3} - \frac{1}{m_{1} m_{2} m_{3} m_{4}} \sum\limits_{j=1}^{4} k_{j}^{3}m_{j}^{2}
\end{align*}
for details (see \cite[Section 4]{CKSTT02}). Although, even after using symmetrization, we are not able to improve the decay for the above $4$-linear estimate for nonresonant frequencies. Although, this symmetrization leads to the cancellation in the resonant case.
\end{rem}

Hence, for the term $I_{11}$, the estimate holds.  For $I_{12},$ we use the symmetrization as follow:
\begin{itemize}
\item[\textbf{Case 1.}]
$k_{2} + k_{3} = 0.$
\item[\textbf{Case 2.}]
$k_{1} + k_{3} = 0.$
\end{itemize}
\textbf{Case 1.} Clearly, we have $k_{2} = - k_{3}$ and $k_{1} = - k_{4}.$ Therefore, from Remark \ref{Symmt}, we have
\begin{align*}
\left[ k^{3}_{1}\left( 1- \frac{m(k_{2} + k_{3} +k_{4})}{m(k_{2})m({k_{3}})m({k_{4}})} \right) \right]_{sym} = \sum\limits_{j=1}^{4} k_{j}^{3} - \frac{1}{m_{1} m_{2} m_{3} m_{4}} \sum\limits_{j=1}^{4} k_{j}^{3}m_{j}^{2},
\end{align*}
which vanishes for $k_{1} = -k_{4}$ and $k_{2} = -k_{3}.$ \\
\textbf{Case 2.} This case is similar to \textbf{Case 1.}

Now, we consider $I_{2}.$ 
From the Fourier transformation, we get
\begin{align*}
I_{2} =& \int\limits_{0}^{T}\int\limits_{0}^{\lambda}\partial_{x}(I'v)^{3}[(I'v)^{3} - I'v^{3}]dx dt, \\
\lesssim & \int\limits_{\substack{\sum\limits_{i=1}^{6}k_{i} =0}} \int\limits_{\sum\limits_{i=1}^{6} \tau_{i} = 0} \bigg|\langle k_{1}  + k_{2} + k_{3} \rangle(\widetilde{I'v_{1}}) (\widetilde{I'v_{2}}) (\widetilde{I'v_{3}}) \\
&\left( 1- \frac{m(k_{4} + k_{5} +k_{6})}{m(k_{4})m({k_{5}})m({k_{6}})} \right) (\widetilde{I'v_{4}})(\widetilde{I'v_{5}})(\widetilde{I'v_{6}})\bigg| (dk_{i})_{\lambda}d\tau_{i},   
\end{align*}
We may suppose $\langle k_{1} \rangle = \max\lbrace \langle k_{i} \rangle, 1 \leq i \leq 3 \rbrace.$ Putting $v = v_{L} + v_{H},$ we divide the integral $I_{2}$ into the following three integrals:
\begin{itemize}
\item[Integral 4.]   
\begin{align*}
\int\limits_{\substack{\sum\limits_{i=1}^{6} k_{i} = 0}} \int\limits_{\sum\limits_{i=1}^{6} \tau_{i} = 0} (\langle k_{1} \rangle \widetilde{I'v_{H}}) (\widetilde{I'v_{L}} + \widetilde{I'v_{H}}) (\widetilde{I'v_{L}} + \widetilde{I'v_{H}})
 \left( 1- \frac{m(k_{4} + k_{5} +k_{6})}{m(k_{4})m({k_{5}})m({k_{6}})} \right)
\\(\widetilde{I'v_{L}})(\widetilde{I'v_{L}})(\widetilde{I'v_{H}}) (dk_{i})_{\lambda}d\tau_{i}.
\end{align*}
\item[Integral 5.] 
\begin{align*}
\int\limits_{\substack{\sum\limits_{i=1}^{6} k_{i} = 0}} \int\limits_{\sum\limits_{i=1}^{6} \tau_{i} = 0} (\langle k_{1} \rangle \widetilde{I'w_{H}}) (\widetilde{I'v_{L}} + \widetilde{I'v_{H}}) (\widetilde{I'v_{L}} + \widetilde{I'v_{H}})
 \left( 1- \frac{m(k_{4} + k_{5} +k_{6})}{m(k_{4})m({k_{5}})m({k_{6}})} \right)
\\(\widetilde{I'v_{H}})(\widetilde{I'v_{H}})(\widetilde{I'v_{L}}) (dk_{i})_{\lambda}d\tau_{i}.
\end{align*}
\item[Integral 6.]  
\begin{align*}
\int\limits_{\substack{\sum\limits_{i=1}^{6} k_{i} = 0}} \int\limits_{\sum\limits_{i=1}^{6} \tau_{i} = 0} (\langle k_{1} \rangle \widetilde{I'v_{H}}) (\widetilde{I'v_{L}} + \widetilde{I'v_{H}}) (\widetilde{I'v_{L}} + \widetilde{I'v_{H}})
 \left( 1- \frac{m(k_{4} + k_{5} +k_{6})}{m(k_{4})m({k_{5}})m({k_{6}})} \right)
\\(\widetilde{I'v_{H}})(\widetilde{I'v_{H}})(\widetilde{I'v_{H}}) (dk_{i})_{\lambda}d\tau_{i}.
\end{align*}
\end{itemize}
\begin{proof}[\textbf{Integral 4.}] \let\qed\relax

Clearly, we have $|k_{4}|,|k_{5}| \ll N'$ and $|k_{6}| \gtrsim N'.$ Hence, the worst condition is $|k_{3}|,|k_{2}| \ll N'$ and $|k_{1}| \gtrsim N'.$ The proof is the same as in $I_{1}.$ From the mean value theorem, we get
\begin{align} \label{second term 1}
\left|\left( 1- \frac{m(k_{4} + k_{5} +k_{6})}{m(k_{4})m({k_{5}})m({k_{6}})} \right)\right|
 \lesssim \frac{|k_{4}| + |k_{5}|}{|k_{6}|}.
\end{align}
We may assume $\langle \sigma_{1} \rangle = \max \lbrace \langle \sigma_{i} \rangle \hspace{1mm}: \hspace{1mm} 1 \leq i \leq 6\rbrace$ as other cases can be treated in the same way. Therefore,
\begin{equation}  \label{second term 2}
\langle \sigma_{1} \rangle^{ 2\epsilon} = \langle \sigma_{1} \rangle^{3\epsilon} \langle \sigma_{1} \rangle^{-\epsilon} \lesssim \langle \sigma_{1} \rangle^{ 3\epsilon} \langle \sigma_{2} \rangle^{-\frac{\epsilon}{2}} \min\lbrace\langle \sigma_{3} \rangle^{-\frac{\epsilon}{2}}, \langle \sigma_{6} \rangle^{-\frac{\epsilon}{2}} \rbrace.
\end{equation}
From Plancherel's theorem, H\"{o}lder's inequality, Proposition \ref{st1}, Lemma \ref{Infinity Estimate} and inequalities (\ref{second term 1}) and (\ref{second term 2}), we get
\begin{align*}
\textit{Integral 4} \lesssim N'^{-1}\|\mathcal{F}^{-1} (\langle \sigma \rangle^{3\epsilon}\langle k_{1} \rangle\widetilde{I'v_{H}})\|_{L^{4}_{x,t}} \|\mathcal{F}^{-1} (\langle \sigma_{2} \rangle^{-\frac{\epsilon}{2}} \widetilde{I'v_{L}})\|_{L^{\infty}_{x,t}} \|\mathcal{F}^{-1}( \langle \sigma_{3} \rangle^{-\frac{\epsilon}{2}} \widetilde{I'v_{L}})\|_{L^{\infty}_{x,t}}  \\
\| \mathcal{F}^{-1}(\langle k_{4} \rangle \widetilde{I'v_{L}})\|_{L^{4}_{x,t}}\|I'v_{L}\|_{L^{4}_{x,t}}\|I'v_{H})\|_{L^{4}_{x,t}} \\
\lesssim N'^{-2}\|I'v_{H}\|_{X^{1,\frac{1}{3}+4\epsilon}}\|I'v_{L}\|_{X^{\frac{1}{2} + \epsilon, \frac{1}{2} - \frac{\epsilon}{2}}} \|I'v_{L}\|_{X^{\frac{1}{2} + \epsilon, \frac{1}{2} - \frac{\epsilon}{2}}} \|I'v_{L}\|_{X^{1, \frac{1}{3} + \epsilon}} \\
\|I'v_{L}\|_{X^{0,\frac{1}{3} +\epsilon}} \|I'v_{H}\|_{X^{1, \frac{1}{3} + \epsilon}}.
\end{align*}
We neglect extra derivatives corresponding to $N_{2},N_{3}$ and $N_{5}$ to get
$$\textit{Integral 4} \lesssim N'^{-2} \|I'v\|^{6}_{X^{1,\frac{1}{2}}}.$$
\end{proof}
\begin{proof}[Integral 5.] \let\qed\relax

Clearly, we have $|k_{4}|,|k_{5}| \gtrsim N'$ and $|k_{6}| \ll N'.$ Hence, the worst condition is $|k_{3}| \ll N'$ and $|k_{1}|,|k_{2}| \gtrsim N'$ as $|k_{1}|$ always have high frequency. From definition of $m$, we get
\begin{align} \label{second term 3}
\left|\left( 1- \frac{m(k_{4} + k_{5} +k_{6})}{m(k_{4})m({k_{5}})m({k_{6}})} \right)\right| \lesssim \left| \frac{m(k_{1})}{m(k_{4})m({k_{5}})} \right|  \lesssim N'^{-1}N_{5}.
\end{align}
From Plancherel's theorem, H\"{o}lder's inequality, Proposition \ref{st1}, Lemma \ref{Infinity Estimate} and inequalities (\ref{second term 2}) and (\ref{second term 3}), we get
\begin{align*}
\textit{Integral 5} \lesssim N'^{-1}\| \mathcal{F}^{-1} (\langle \sigma \rangle^{3\epsilon}\langle k_{1} \rangle \widetilde{I'v_{H}})\|_{L^{4}_{x,t}} \|I'v_{H}\|_{L^{4}_{x,t}} \| \mathcal{F}^{-1} (\langle \sigma_{3} \rangle^{-\frac{\epsilon}{2}} \widehat{I'v_{L}})\|_{L^{\infty}_{x,t}} \|I'v_{H}\|_{L^{4}_{x,t}} \\
\| \mathcal{F}^{-1} (\langle k_{5} \rangle \widetilde{I'v_{H}})\|_{L^{4}_{x,t}} \| \mathcal{F}^{-1} (\langle \sigma_{6} \rangle^{-\frac{\epsilon}{2}}\widetilde{I'v_{H}})\|_{L^{4}_{x,t}} \\
\lesssim N'^{-1}\|I'v_{H}\|_{X^{1,\frac{1}{3}+4\epsilon}}\|I'v_{H}\|_{X^{0, \frac{1}{3} + \epsilon}} \|I'v_{L}\|_{X^{\frac{1}{2} + \epsilon, \frac{1}{2} - \frac{\epsilon}{2}}} \|I'v_{H}\|_{X^{0, \frac{1}{3} + \epsilon}} \\
\|I'v_{H}\|_{X^{1,\frac{1}{3} +\epsilon}}\|I'v_{L}\|_{X^{\frac{1}{2} + \epsilon, \frac{1}{2} - \frac{\epsilon}{2}}}.
\end{align*}
We neglect extra derivatives corresponding to $N_{3}$ and $N_{6}$ to get
$$\textit{Integral 4} \lesssim N'^{-3} \|I'v\|^{6}_{X^{1,\frac{1}{2}}}.$$
\end{proof}

\begin{proof}[Integral 6.] \let\qed\relax

Clearly, we have $|k_{4}|,|k_{5}|, |k_{6}| \gtrsim N'.$ Hence, the worst condition is $|k_{3}|, |k_{2}| \ll N'$ and $|k_{1}| \gtrsim N'.$. From definition of $m$, we get
\begin{align} \label{second term 4}
\left|\left( 1- \frac{m(k_{4} + k_{5} +k_{6})}{m(k_{4})m({k_{5}})m({k_{6}})} \right)\right| \lesssim \left| \frac{m(k_{1})}{m(k_{4})m(k_{5})m(k_{6})} \right|  \lesssim N'^{-2}|k_{5}||k_{6}|.
\end{align}
From Plancherel's theorem, H\"{o}lder's inequality, Proposition \ref{st1}, Lemma \ref{Infinity Estimate} and inequalities (\ref{second term 2}) and (\ref{second term 4}), we get
\begin{align*}
\textit{Integral 6} \lesssim N'^{-2}\| \mathcal{F}^{-1} (\langle \sigma \rangle^{3\epsilon}\langle k_{1} \rangle \widetilde{I'v_{H}})\|_{L^{4}_{x,t}} \| \mathcal{F}^{-1} (\langle \sigma_{2} \rangle^{-\frac{\epsilon}{2}} \widetilde{I'v_{L}})\|_{L^{\infty}_{x,t}} \| \mathcal{F}^{-1} (\langle \sigma_{3} \rangle^{-\frac{\epsilon}{2}} \widetilde{I'v_{L}})\|_{L^{\infty}_{x,t}} \\
\|\mathcal{F}^{-1} (\langle k_{4} \rangle \widetilde{I'v_{H}})\|_{L^{4}_{x,t}} \| \mathcal{F}^{-1} (\langle k_{5}\rangle \widetilde{I'v_{H}})\|_{L^{4}_{x,t}}\|I'v_{H})\|_{L^{4}_{x,t}} \\
\lesssim N'^{-2}\|I'v_{H}\|_{X^{1,\frac{1}{3}+4\epsilon}}\|I'v_{L}\|_{X^{\frac{1}{2} + \epsilon, \frac{1}{2} - \frac{\epsilon}{2}}} \|I'v_{L}\|_{X^{\frac{1}{2} + \epsilon, \frac{1}{2} - \frac{\epsilon}{2}}} \|I'v_{H}\|_{X^{1, \frac{1}{3} + \epsilon}} \\ 
\|I'v_{H}\|_{X^{1,\frac{1}{3} +\epsilon}}\|I'v_{H}\|_{X^{0, \frac{1}{3} + \epsilon}}.
\end{align*}

We neglect extra derivatives corresponding to $N_{2}$ and $N_{3}$ to get
$$\textit{Integral 4} \lesssim N'^{-3 } \|I'v\|^{6}_{X^{1,\frac{1}{2}}}.$$
\end{proof}
\begin{rem}
Note that the sexalinear term does not depend on the scaler parameter $\lambda$.
\end{rem}

\begin{center}
\section*{Appendix}
\end{center}
The following example is given by Prof. Nobu Kishimoto which explain why we need to use the inhomogeneous Soblev norm in place of homogeneous norm. In fact, for homogeneous norm the Proposition \ref{TL Main result} does not hold.
Define the space $\dot X^{s,\frac{1}{2}}$ via the norm
$$\|u\|_{\dot X^{s,\frac{1}{2}}} = \||k|^{s}\langle \tau - 4\pi^{2}k^{3} \rangle^{b}\tilde{u}(k,\tau)\|_{L^{2}((dk)_{\lambda},d\tau)}.$$
\begin{eg}
Assume $\lambda \geq 1$ and $\sqrt{\lambda} \in \mathbb{Z}/\lambda.$ Let $\lambda \mathbb{T} = \mathbb{R}/\lambda \mathbb{Z}.$ We define the functions $v_{1},v_{2},v_{3}$ on $\lambda\mathbb{T} \times \mathbb{R}$ by
\begin{align*}
\tilde{v}_{1}(k,\tau) &= 1_{[-1,1]}(\tau - 4\pi^{2}k^{3})\cdot1_{\{1/\lambda\}}(k), \\
\tilde{v}_{2}(k,\tau) &= 1_{[-1,1]}(\tau - 4\pi^{2}k^{3})\cdot1_{\{-2/\lambda\}}(k), \\
\tilde{v}_{3}(k,\tau) &= 1_{[-1,1]}(\tau - 4\pi^{2}k^{3})\cdot1_{\{\sqrt{\lambda}\}}(k).
\end{align*}
We have
\begin{align*}
\|v_{1}\|_{\dot X^{s,\frac{1}{2}}} \sim \|v_{2}\|_{\dot X^{s,\frac{1}{2}}} \sim \left(\frac{1}{\lambda}\right)^{s} \lambda^{-\frac{1}{2}} = \lambda^{s-\frac{1}{2}}, \\
 \|v_{3}\|_{\dot X^{s,\frac{1}{2}}} \sim (\sqrt{\lambda})^{s} \lambda^{-\frac{1}{2}} = \lambda^{\frac{s}{2}-\frac{1}{2}}.
\end{align*}
We see that 
\begin{align*}
&\left|\tilde{J}[v_{1},v_{2},v_{3}](\sqrt{\lambda})-\frac{1}{\lambda},\tau)\right| \\
&\sim \sqrt{\lambda}\left|\int_{\tau_{1} + \tau_{2} + \tau_{3} = \tau}\int_{\substack{k_{1} + k_{2} + k_{3} = \sqrt{\lambda} - \lambda^{-1} \\ (k_{1} + k_{2})(k_{2} + k_{3})(k_{3} + k_{1}) \neq 0}} \prod\limits_{j=1}^{3} \tilde{v}_{j}(k_{j},\tau_{j})(dk_{1})_{\lambda}(dk_{2})_{\lambda}d\tau_{1}d\tau_{2} \right| \\
&\gtrsim \lambda^{-3/2} 1_{[-1,1]}(\tau -4\pi^{2}(\sqrt{\lambda}-\lambda^{-1})^{3} + 4\pi^{2}M),
\end{align*}
where
$$M = 3\left(\frac{1}{\lambda} + \frac{-2}{\lambda} \right)\left(\frac{-2}{\lambda} + \sqrt{\lambda} \right)\left(\sqrt{\lambda} + \frac{1}{\lambda}  \right),$$
so that $|M| \sim 1.$ Hence, we have 
$$\|J[v_{1},v_{2},v_{3}]\|_{\dot X^{s,\frac{1}{2}}} \gtrsim \lambda^{-\frac{3}{2}}\cdot (\sqrt{\lambda})^{s}\lambda^{-\frac{1}{2}} =\lambda^{\frac{s}{2} - 2}.$$
Therefore, if the trilinear estimate 
$$\|J[v_{1},v_{2},v_{3}]\|_{\dot X^{s,\frac{1}{2}}} \lesssim \lambda^{0+} \|v_{1}\|_{\dot X^{s,\frac{1}{2}}} \|v_{2}\|_{\dot X^{s,\frac{1}{2}}} \|v_{3}\|_{\dot X^{s,\frac{1}{2}}}$$
were true, it would imply that
$$\lambda^{\frac{s}{2} - 2} \lesssim (\lambda^{-s-\frac{1}{2}})^{2}\lambda^{\frac{s}{2} - \frac{1}{2}} \hspace{4mm}\Leftrightarrow \hspace{2mm} \lambda^{2s} \lesssim \lambda^{\frac{1}{2}+} \hspace{2mm} (\lambda \geq 1).$$
For large $\lambda,$ this holds only if $s \leq \frac{1}{4}+.$
\end{eg}

\end{proof}
\section*{Acknowledgements}
The author would like to express his deep gratitude to Professor Yoshio Tsutsumi for giving him valuable suggestions and constant encouragement. The author is also greatful to Professors Kotaro Tsugawa and Nobu Kishimoto and Mr. Minjie Shan for fruitful discussions. Finally, the author is indebted to the referee for his or her valuable remarks.

\end{document}